\documentclass{article}
\usepackage[T1]{fontenc}
\usepackage[utf8]{inputenc}
\usepackage[english]{babel}
\usepackage[margin=3cm]{geometry}
\usepackage{amsmath, amssymb}
\usepackage{amsfonts}
\usepackage{dsfont}
\usepackage{float}
\usepackage{graphicx}
\usepackage{wrapfig}
\usepackage{mathtools}
\usepackage{bbm}
\usepackage{amsthm}
\usepackage{ifthen}
\usepackage{graphicx}
\usepackage{hyperref}
\usepackage[ruled,vlined]{algorithm2e}
\usepackage{xcolor}
\usepackage{mathtools}
\usepackage{empheq}

\newtheorem{theorem}{Theorem}[section]
\newtheorem{prop}[theorem]{Proposition}
\newtheorem{definition}[theorem]{Definition}
\newtheorem{remark}{Remark}[section]
\newtheorem{lemma}[theorem]{Lemma}
\newtheorem{Cl}{Claim}[section]

\newcommand{\ca}[1]{\mathcal{#1}}
\newcommand{\bb}[1]{\mathbb{#1}}
\newcommand{\p}{\mathbb{P}}

\newcommand{\comillas}[1]{``#1''}
\newcommand{\set}[1]{\left\{#1\right\}}
\newcommand{\parent}[1]{\left(#1\right)}

\newcommand{\Rd}{\bb{R}^d}
\newcommand{\R}{\bb{R}}

\newcommand{\ind}[1]{\mathbbm{1}_{#1}}
\newcommand{\esp}[1]{\bb{E}\barras{#1}}

\newcommand{\gb}[1]{\overline{\widehat{#1}}}
\newcommand{\barras}[1]{\left| #1 \right|}
\newcommand{\integral}{\int_{t_i}^{t_{i+1}}}
\newcommand{\ug}[1]{\widehat{\ca{U}}_{#1}}
\newcommand{\vg}[1]{\widehat{\ca{V}}_{#1}}
\newcommand{\norm}[1]{\left\lVert#1\right\rVert}
\newcommand{\xscheme}[1]{X_{t_{#1}}^{\pi}}
\newcommand{\prom}[1]{\langle #1 \rangle}

\begin{document}

\title{Deep Learning Schemes For Parabolic Nonlocal Integro-Differential Equations}
\author{Javier Castro\footnote{address: Departamento de Ingenier\'{\i}a Matem\'atica, Universidad de Chile, Casilla
170 Correo 3, Santiago, Chile. email: jcastro@dim.uchile.cl}
\footnote{J.C.'s work was funded in part by Chilean research grants FONDECYT 1191412, and CMM Conicyt PIA AFB170001.}}

\numberwithin{equation}{section}

\maketitle

\begin{abstract}
In this paper we consider the numerical approximation of nonlocal integro differential parabolic equations via neural networks. These equations appear in many recent applications, including finance, biology and others, and have been recently studied in great generality starting from the work of Caffarelli and Silvestre \cite{CS}. Based in the work by Hure, Pham and Warin \cite{DBS}, we generalize their Euler scheme and consistency result for Backward Forward Stochastic Differential Equations to the nonlocal case. We rely on L\`evy processes and a new neural network approximation of the nonlocal part to overcome the lack of a suitable good approximation of the nonlocal part of the solution.
\end{abstract}

\tableofcontents
\section{Introduction}

A difficult problem in Applied Mathematics is to approximate solutions of Partial Differential Equations (PDEs) in large dimensions. In low dimensions such as $1,2$ or $3$, classical methods such as finite differences or finite elements are commonly applied, with satisfactory convergence orders (see e.g. Allaire \cite[Chapters 2 and 6]{allaire}). An important problem appears when dealing with high dimensional problems such as \emph{portfolio management}, where each dimension represents the size of some financial derivative in the portfolio. More complications appear when the PDE is nonlocal, as present in many applications. For finite difference methods, one needs to construct a mesh that, computationally speaking, has exponential cost on the dimension $d\in \bb{N}$ of the considered PDE. This problem is known in the literature as \emph{the curse of dimensionality}, and the most common attempt to solve this issue is via stochastic methods. Deep Learning (DL) methods have proven to be an efficient tool to handle this problem and to approximate solutions of high dimensional second order fully nonlinear PDEs. This is achieved by finding that the solution of the PDE, evaluated at some diffusion process, solves an Stochastic Differential Equation (SDE); then an Euler scheme together with DL is applied to solve the SDE, see \cite{intro5,DBS} for key developments.\\

Without being exhaustive, we present some of the current developments in this direction. First of all, Monte Carlo algorithms are an important approach to the resolution of this dimensional problem. This can be done by means of the classical Feynman-Kac representation, that allows us to write the solution of a linear PDE as an expected value, and approximate high dimensional integrals with an average over simulations of some random variables. The key developments in this area can be found in Han-Jentzen-E \cite{HJE} and Beck-E-Jentzen \cite{intro5}. On the other hand, Multilevel Picard method (MLP) is another approach and consist on interpreting the stochastic representation of the solution to a semilinear
parabolic (or elliptic) PDE as a fixed point equation. Then, by using Picard iterations together with Monte Carlo methods for calculating some important integrals, one is able to approximate the solution to the PDE, see \cite{intro1, intro2} for fundamental advances in this direction. On the other hand, the so-called Deep Galerkin method (DGM) is a DL approach used to solve quasilinear parabolic PDEs plus boundary and initial conditions. The cost function in this framework is defined in an intuitive way, and consists of the differences between the approximation solution $\hat u$ evaluated at the initial time and spatial boundary, with the true initial and boundary conditions, plus the value of the equation evaluated at $ \hat u$. These quantities are captured by an $L^2$ norm, which in high dimensions is minimized using the stochastic gradient descent method. See \cite{intro3} for the development of the DGM and \cite{intro4} for an application.  \\

In \cite{DBS}, the principal source of inspiration of this article, Hure, Pham, and Warin consider the framework introduced previously in \cite{intro5} and present new approximation schemes for the solution of a parabolic nonlinear PDE and its gradient via  Neural Networks. Via an intricate use of intermediate numerical approximations for each term in their scheme, they prove the numerical consistency and high accuracy of the method, at least in the case of low dimensions. \\

The goal of this article is to deal with the \emph{curse of dimensionality} problem in PDEs of \emph{integral, nonlocal type}. We call them PIDE models. In general, standard PDEs model situations where, in order to know the state of a system at a particular point, one needs information of the state in a arbitrarily small neighborhood of the point. On the contrary, PIDEs can model more general phenomena where \emph{long distance} interactions and effects are important and must be considered. An important example of PIDEs are those which involve fractional derivatives, such as the \emph{Fractional Laplacian}. This operator has been extensively studied, from the PDE point of view, during the past ten years, starting from the fundamental work by Caffarelli and Silvestre \cite{CS}. See \cite{DPV,fractional-operator} and references therein for nice introductions to this operator, one of the most relevant examples of integro-differential operators. More generally speaking, nonlocal equations are used in a wide range of scientific areas, see \cite{advection-dispersion} for applications in advection dispersion equations,  \cite{image-processing} for image processing, \cite{perodynamic} for perodyinamic, \cite{hydrodynamics} for hydrodynamics, and see \cite{asset-returns, financial-app} for finances. For more theoretical results on nonlocal equations, see e.g. \cite{erwin1, erwin2, erwin3} and references therein. In \cite{non-local-guide}, the authors give a complete introduction to nonlocal equations and then they develop nonlocal version of three numerical methods: finite difference, finite element and Spectral-Galerkin.\\

We present here an extension and generalization of \cite{DBS} to PIDEs, by adding nonlocal contribution to the PDE. Some important changes are needed in the algorithm, including the use of a third Neural Network to approximate nonlocal parts of the solution. Of particular utility will be the result shown in \cite{bruno} to prove convergence of numerical schemes. As far as we know, this is the first result of neural networks applied to PIDEs, but still incomplete, as we will see below. 
\\

The basic idea of the Euler scheme presented in this article is based on that presented by Zhang in \cite{zha04}. In that paper, the author gives a discrete time approximation of a BSDE (backward SDE) with no jump terms. The scheme involves the computation of conditional expectations and gives important bounds and results that were used in \cite{DBS} to prove the convergence of a DL algorithm to solve a second order fully nonlinear PDE. In our case, nonlocal integral models require additional treatments. The work by Bouchard and Elie \cite{bruno}, very important for the work presented here, generalizes the properties given in \cite{zha04} to the nonlocal setting by considering Lévy process. We will closely follow their approach to construct our numerical scheme. In \cite{torres}, the authors present a discrete-time approximation of a BSDEJ (BSDE with jump terms) such that its solution converges weakly to a solution of the continuous in time equation. They also use this method to approximate the solution to the correspondent PIDE.\\


\subsection{Setting} Let $d\geq 1$ and $T>0$. Consider the following integro-differential PDE
\begin{equation}
    \left \{
    \begin{aligned}
        \ca{L}u(t,x) + f(t,x,u(t,x),\sigma (x) \nabla u(t,x),\ca{I}[u](t,x))&=0, && (t,x)\in[0,T]\times\Rd ,\\
        u(T,x) &= g(x), && x\in\Rd.
    \end{aligned}
    \right.
    \label{eq: pide}
\end{equation}
Here, $u=u(t,x)$ is the unknown of the problem. For a positive number $t$, let $I_t=[0,t]$. The operator $\ca{L}$ above is of parabolic nonlocal type, and is defined, for $u\in \ca{C}^{1,2}(I_T\times\Rd)$, as follows:
\begin{equation}\label{mathcalL}
\begin{aligned}
    \ca{L}u(t,x) =  \partial_t u(t,x)+ \nabla u(t,x)\cdot b(x) &+ \frac{1}{2}\nabla\cdot(\sigma(x)\sigma (x)^T\nabla u(t,x))\\
    &+ \int_{\Rd} [u(t,x+\beta(x,y))-u(t,x)-\nabla u(t,x)\cdot\beta(x,y)]\lambda (dy),
\end{aligned}
\end{equation}
where $\lambda(dy)$ is a finite measure on $\Rd$, equipped with its Borel $\sigma$-algebra, and a Lévy measure as well which means that
\[
    \lambda(\set{0}) =0\qquad \text{and} \qquad \int_{\Rd} (1\wedge |y|^2)\lambda(dy) < \infty.
\]
Also, $f:I_T\times\Rd\times\R\times\Rd\times \R\to\R$. We also assume the standard Lipschitz conditions on the functions in order to have a unique solution to \eqref{eq: pide} in the class $C^{1,2}$: there exists a universal constant $K>0$ such that
\begin{equation}\label{CC}
\textbf{(C)}\;
\begin{cases}
\hbox{$\bullet$ (Regularity) $g:\Rd\to\R$, $b:\Rd\to\Rd$ and $\sigma:\Rd\to\bb{R}^{d\times d}$ are $K$-Lipschitz real, vector}\\
\hbox{\quad and matrix valued functions, respectively.}\\
\hbox{$\bullet$ (Boundedness)$\beta:\Rd\times\Rd\to\Rd$ and  $\sup_{y\in\Rd} |\beta(0,y)|\le K$.}\\
\hbox{$\bullet$ (Uniformly Lipschitz) $\sup_{y\in\Rd}|\beta(x,y)-\beta(x',y)|\le K|x-x'|,\ \forall\ x,x'\in\Rd$.}\\
\hbox{$\bullet$ (H\"older continuity) For each $t,t',y,y',w,w'\in\mathbb R$ and $x,x',z,z' \in\Rd$, one has}\\
\hbox{\quad  $|f(t,x,y,z,w)-f(t',x',y',z',w')|\le K\big(|t-t'|^{1/2}+|x-x'|+|y-y'|+|z-z'|+|w-w'|\big)$.}\\
\hbox{$\bullet$ (Invertibility) For each $y\in\Rd$, the map $x\to\beta(x,y)$ admits a Jacobian matrix $\nabla\beta(x, y)$}\\
\hbox{\quad such that the function $a(x,\xi; y)=\xi^T (\nabla\beta(x, y) + I)\xi$ satisfies, for all $x,y \in \Rd$, }\\
\hbox{\quad $a(x,\xi;y)\ge |\xi|^2 K^{-1}$ or $a(x,\xi;y)\le -|\xi|^2 K^{-1}$.}
\end{cases}
\end{equation}
The last condition is of technical type and it is needed to ensure the validity of certain approximation results (see Proposition \ref{chicos}). On the other hand, the nonlocal, integro-differential operator $\ca{I}$ is defined as
\begin{equation}\label{cI}
    \ca{I}[u](t,x) = \int_{\Rd} \big(u(t,x+\beta(x,y))-u(t,x) \big) \lambda(dy).
\end{equation}
The conditions stated in \eqref{CC} are standard in the literature (see \cite{pardoux,bruno,torres}) and are needed to ensure the existence and uniqueness (with satisfactory bounds mentioned below) of solutions to a FBSDE (forward BSDE) related to \eqref{eq: pide}.

\subsection{Forward Backward formulation of \eqref{eq: pide}}
In the previous context, for $t\in [0,T]$, consider the following stochastic setting for \eqref{eq: pide}. Let $(\Omega,\ca{F},\bb{F},\p)$, $\bb{F}=(\ca{F}_t)_{0\le t\le T}$, be a stochastic basis satisfying the usual conditions: $\bb{F}$ is right continuous, and $\mathcal F_0$ is complete (contains all zero measure sets). The filtration $\bb{F}$ is generated by a $d$-dimensional Brownian motion (BM) $W=(W_t)_{0\le t\le T}$ and a Poisson random measure $\mu$ on $\bb{R}_+\times \Rd,$ independent of $W$. Let $L^p:=L^p (\Omega,\ca{F},\p)$ the space of random variables with finite $p$ moment.\\

Recall that $\lambda(dy)$ is a finite Lévy measure on $\Rd$. The \emph{compensated measure} is denoted as 
\begin{equation}\label{compensated}
\overline{\mu}(dt,dy)=\mu(dt,dy) - \lambda(dy)dt, 
\end{equation}
and is such that for every measurable set $A$ satisfying $\lambda(A)<\infty$, $(\overline{\mu}(t,A):=\overline{\mu}([0,t],A))_t$ is a martingale. Given a time $t_i\in [0,T]$, the operator $\bb{E}_i$ will denote the conditional expectation with respect to $\ca{F}_{t_i}$:
\begin{equation}\label{Esperanza_condicional}
    \bb{E}_i\parent{X} := \bb{E}\parent{X\big| \ca{F}_{t_i}}.
\end{equation}
Recall the equation \eqref{eq: pide}-\eqref{mathcalL}-\eqref{cI}. As usual, $X_{r^-}$ denotes the a.e. limit of $X_s$ as $s\uparrow r$. Let us consider the next forward and backward stochastic differential equations with jumps in terms of the unknown variables $(X_t,Y_t,Z_t, U_t)$: 
\begin{align}
    X_t&=x+\int_0^t b(X_s)ds+\int_0^t\sigma(X_{s^-})\cdot dW_s+\int_0^t\int_{\Rd}\beta(X_{s^-},y)\overline{\mu}(ds,dy),
    \label{eq:fbsdej}\\
    Y_t&=g(X_T)+\int_t^T f(\Theta_s)dr-\int_t^T Z_s\cdot dW_s -\int_t^T\int_{\Rd}U_s(y)\overline{\mu}(ds,dy) ,
    \label{eq:bsdej}\\
    \Gamma_t &=\int_{\Rd}U_t (y)\lambda(dy),
    \label{eq:gamma}
\end{align}
where $\Theta_s=(s,X_s,Y_s,Z_s,\Gamma_s)$ for $0\le s\le T$ and $x\in\Rd$. Note that $Z_s$ is vector valued. \\

By applying Itô's lemma (see \cite[Thm 2.3.4]{lukasz}) to the solution $X_t$ in \eqref{eq:fbsdej}  and a $\ca{C}^{1,2}(I_T\times\Rd)$ solution $u$ of PIDE (\ref{eq: pide}) as $Y_t$ in \eqref{eq:bsdej}, we obtain the compact stochastic formulation of \eqref{eq: pide}:
\begin{align}\label{Ito_key}
    u(t,X_t)= & ~{} u(0,X_0)-\int_0^t f(s,X_{s^-},u(s,X_{s^-}),\sigma(X_{s^-})\nabla u(s,X_{s^-}), \ca{I}[u](s,X_{s^-})) ds,\\
&+\int_0^t [\sigma(X_{s^-}) \nabla u(s,X_{s^-})]\cdot dW_s+\int_0^t\int_{\Rd}[u(s,X_{s^-}+\beta(X_{s^-},y))-u(s,X_{s^-})]\overline{\mu}(ds,dy),\nonumber
\end{align}
valid for $t\in [0,T]$. This tells us that whatever we use as approximations of
\begin{align*}
    u(t,X_t),\qquad \sigma(X_t)\nabla u(t,X_t)\qquad \text{and}\qquad u(t,X_{t}+\beta(X_{t},\cdot))-u(t,X_{t}),
\end{align*}
must satisfy (\ref{Ito_key}) in some proper metric. 
An important statement here is that the conditions (\ref{CC}) ensure the existence of a \textit{viscosity solution} $u\in\ca{C}(I_T\times\Rd)$ with at most polynomial growth such that $u(t,X_t) = Y_t$, and this is why our scheme deals with solving the FBSDEJ, see \cite[Thm 3.4]{pardoux}. In order to present the algorithm to approximate this last equation via NNs, we first need to introduce them. The following section may be taken independent of the rest of the paper. The reader familiarized with NNs can advance immediately to Section \ref{Discretization}. 

\subsection*{Organization of this paper} The rest of this work is organized as follows. Section \ref{NNNs} recalls the main results on Neural Networks needed in this paper. In Section \ref{Discretization} we introduce the discretization scheme for PIDEs. In Section \ref{prelim} we state all the preliminary results needed in this paper for the proof of Theorem \ref{MT1}. Section \ref{sect:Proof} contains the proof of Theorem \ref{MT1} and Subection \ref{sect:opti} studies the optimization of our algorithm.

\section{Neural Networks and Approximation Theorems}\label{NNNs} 

Neural Networks (NN) are not recent. In \cite{pitts} and \cite{perceptron}, published in $1943$ and $1958$ respectively, the authors introduce the concept of a NN but far from the actual definition. Through the years, the use of a NN as function approximates started to gain importance for its well performance in applications. A rigorous justification of this property was proven in \cite{Hornik, Leshno}, using the Stone-Weierstrass theorem. These papers state that the good performance of neural networks is not a fortuitous result, but a well established and justified property. See \cite{origin-dl, state-art-dl} for a review on the origin and state of the art survey of DL, respectively.\\

The huge amount of available data, due to social media, astronomical observatories and even Wikipedia, together with the progress of computational power, have allowed us to train more and more efficient Machine Learning (ML) algorithms, considering data that years ago were not possible to analyze. \emph{Deep Learning} is a part of supervised ML algorithms and it concerns with the problem of approximating an unknown nonlinear function $f:X\to Y$, where $X$ represents the set of possibles inputs and $Y$ the outputs, for example $Y$ could be a finite set of classes and therefore $f$ has a classification task. In order to perform a DL algorithm, we need a data set of the form $D = \set{(x,f(x)): x\in A}$ with $A\subset X$, which in the literature is also known as the \emph{training set}. The next step is to define a family of candidates $\set{f_\theta: \theta\in\Theta}$ of functions parametrized by $\theta\in\Theta$. Now, with this set up, the final step is to find an optimal $\theta^*\in\Theta$ minimizing some proper cost function $L(\theta;D)$ over $\Theta$. The definition given is too general and leaves a lot of questions on how to implement a DL algorithm; some of these questions will be answered whenever we give a formal definition of a NN.\\

The complexity and generality of the main problem that DL is trying to solve, makes it useful to a large variety of disciplines in science. In astronomy, the large amount of data recollected by observatories makes it a perfect place to implement ML, see \cite{astronomy-app1} for a review of ML in astronomy and \cite{light-curves} for a concrete use of Convolutional Neural Networks (CNN) to classify light curves. See \cite{physics-dl1} for a review of ML on experimental high energy physics and \cite{qst} for an application of NN on quantum state tomography. In \cite{street-style}, the authors use DL to find patterns in fashion and style trends by space and time, using data from Instagram. In \cite{brain-tumor} the authors train a CNN to classify brain tumors into Glioma, Meningioma, and Pituitary Tumor reaching a high accuracy. See \cite{medical-survey} for a survey on the use of DL in medical science where CNN are the most common type of DL structure.\\

To fix ideas, in this paper we focus on a simpler setting, where the inputs and outputs spaces are multidimensional real spaces. In order to define the candidate functions we need an input dimension $d$, a number of layers $L$ with $l_i$ neurons, each for $i\in\set{1,...,L}$, an output dimension $k = l_L$, weight matrix $(W_i)_{i=1}^L$, a bias vectors $(b_i)_{i=1}^L$, and an activation function $\sigma:\R\to\R$. The activation function is a way to break the linearity.

\begin{definition}
    Given $(d,L,l_i,W_i,b_i)$ as above, $\theta=(W_i,b_i)$, we define the neural network $\ca{U}:\Rd\to\R^{l_L}$ as the following composition
    \begin{align*}
        \ca{U}(x;\theta)=\parent{ A_L\circ\sigma\circ A_{L-1}\circ\cdots \circ A_2 \circ \sigma\circ A_1}(x),
    \end{align*}
    where $A_i:\R^{l_{i-1}}\to\R^{l_i}$ is the affine linear function such that $A_i(x)=W_ix+b_i$ and $\sigma$ is applied component-wise.
\end{definition}

In the following, the input and output dimensions will be fixed parameters. The range of functions that we can compute varying the dimensions of the parameters $\theta=(W_i,b_i)$ will be called the space of neural networks and will be denoted by $\ca{N}$.
The next theorem can be found on \cite{Leshno} as \comillas{Hornik Theorem 1}.

\begin{theorem}[\cite{Hornik, Leshno}]\label{uat}
If the activation function is bounded and nonconstant, then the neural network space $\ca{N}$ is dense in $L^p (\mu)$ for every finite measure $\mu$ in $\Rd$.
\end{theorem}

This theorem tells us that if we want to approximate, for example, some function $f:\Rd\to \R$ in $L^2$, the quantity 
\begin{equation}\label{NN_L2norm}
    \underset{\xi}{\text{inf}} \int_{\Rd} (\ca{U}(x;\xi) - f(x))^2 \mu(dx)
\end{equation}
can be made arbitrarily small by possible making the dimension of the parameters growing sufficiently large, whenever $\mu$ is a finite measure on $\Rd$ and the activation function that defines the NN is bounded and non-constant.\\

\section{Discretization of the dynamics}\label{Discretization}

Fix a constant step partition of the interval $I_T$, defined as $\pi=\set{\frac{iT}{N}}_{i\in\set{0,...,N}}$, $t_i= \frac{iT}{N}$, and set $\Delta W_i=W_{t_{i+1}}-W_{t_i}$. Also, define $h:=\frac{T}N$ and (with a slight abuse of notation),  $\Delta t_i=(t_i,t_{i+1}]$. Recall the compensated measure $\overline\mu$ from \eqref{compensated}. Let 
\begin{equation}
M_t=\overline{\mu}((0,t],\Rd) \quad \hbox{and} \quad \Delta M_i=\overline{\mu}((t_i,t_{i+1}],\Rd):=\int_{t_i}^{t_{i+1}}\!\!\int_{\Rd}\overline{\mu}(ds,dy).
\end{equation}
It is well-known that an Euler scheme for the first equation in (\ref{eq:bsdej}) obeys the form 
\begin{align}
    X^{\pi}_0 &= x,
    \label{eq:euler-x1}\\
    X^{\pi}_{t_{i+1}}&=X^{\pi}_{t_i}+h \, b(X^{\pi}_{t_i})+\Delta W_i\sigma(X^{\pi}_{t_i})+\int_{\Rd}\beta(X_{t_i},y)\overline{\mu}((t_i,t_{i+1}],dy).
    \label{eq:euler-x2}
\end{align}
Note that due to the finiteness of the $\lambda$ we don't care much for the discontinuities as there are a finite number of those. This scheme satisfies the next error bound (\cite{bruno})
\begin{align}\label{error_X}
    \max_{i=1,...,N}\bb{E}\parent{\sup_{t\in [t_i,t_{i+1}]} |X_t-X^{\pi}_{t_i}|^2} = O(h).
\end{align}
Adapting the argument of \cite{DBS} to the nonlocal case, and in view of \eqref{Ito_key}, we propose the following modified Euler scheme: for $i=0,1,\ldots,N$,
\begin{align*}
    u(t_{i+1},X^\pi_{t_{i+1}})\approx &~{} F_i\Big( t_i,X^\pi_{t_i},u(t_i,X^\pi_{t_i}),\sigma(X^\pi_{t_i}) \nabla u(t_i,X^\pi_{t_i}) ,u(t_i,X^\pi_{t_i}+\beta(X^\pi_{t_i},\cdot))-u(t_i,X^\pi_{t_i}),h,\Delta W_i \Big),
\end{align*}
%
where $F_i:I_T\times\Rd\times\R\times\Rd\times L^1(\lambda)\times\bb{R}_+\times\Rd\to\R$ is defined as
\begin{align*}
    F_i(t,x,y,z,g,h,w):= y - hf \left(t,x,y,z,\int_{\Rd}g(y)\lambda(dy) \right) + w\cdot z + \int_{\Rd} g(y)\bar{\mu}\parent{(t_i,t_{i+1}],dy}.
\end{align*}

\begin{remark}\label{remark-mc}
Note that the non local term in (\ref{eq: pide}) forces us to define $F_i$ in such a way that its fifth argument must be a function $g$ in $L^1 (\lambda)$. In view if the integrals involved in $F_i$, it appears that we are again facing the same high dimensional problem; however this problem may be instead treated with Monte Carlo approximations, see below. 
\end{remark}

\begin{remark}
In the nonlocal setting, the function $F_i$ also depends on the step in terms of the integrated measure $\bar{\mu}\parent{(t_i,t_{i+1}],dy}$. This is an important change in the Euler scheme, since we do not approximate the nonlocal term at time $t_i$ in this case, but instead take into account the whole measure $\bar \mu$ of the time interval $(t_i,t_{i+1}]$. 
\end{remark}

Recall Theorem \ref{uat}. For every time $t_i$ on the grid, along the proof we will choose NNs 
\begin{equation}\label{NNs}
\big(\ca{U}_i(\cdot;\theta),\ca{Z}_i(\cdot;\theta),\ca{G}_i(\cdot,\circ;\theta)\big)
\end{equation}
approximating in some sense to be specified 
 \[
 (u(t_i,\cdot),~\sigma(\cdot) \nabla u(t_i,\cdot), ~u(t_i,\cdot + \beta(\cdot,\circ))-u(t_i,\cdot)),
 \]
respectively. Let also
\begin{align}\label{calG}
\prom{\ca{G}}_i(x;\theta) = \int_{\mathbb R^d}\ca{G}_i(x,y;\theta)\lambda(dy).
\end{align}
We propose an extension of the DBDP1 algorithm presented on \cite{DBS}. The idea of the algorithm is that the NN, evaluated on $\xscheme{i}$, are good approximations of the processes solving the FBSDEJ.
Let $L_i$ be a cost function defined on a parameters space as
\begin{align}\label{eq:loss}
    L_i(\theta)=\bb{E}\left|\ug{i+1}(X^\pi_{t_{i+1}})-F(t_i,X^\pi_{t_i},\ca{U}_i(X^\pi_{t_i};\theta),\ca{Z}_i(X^\pi_{t_i};\theta),\ca{G}_i(X^\pi_{t_i},\cdot;\theta),h,\Delta W_i)\right|^2.
\end{align} 
\begin{algorithm}[H]
\label{algorithm1}
\SetAlgoLined
Start with $\widehat{\ca{U}}_N(\cdot)=g(\cdot)$\;
\For{$i\in\set{N-1,...,1}$}{
  Given $\ug{i+1}$\;
  Minimize $\theta\to L_i(\theta)$\;
  Update $(\ug{i},\widehat{\ca{Z}}_{i},\widehat{\ca{G}}_{i})=(\ca{U}_i(\cdot;\theta^*),\ca{Z}_i(\cdot;\theta^*),\ca{G}_i(\cdot,\circ;\theta^*))$\;
 }
\caption{DBDP1 PIDE extension}
\end{algorithm}
\vspace{.5cm}

For the minimization step we need to calculate an expected value, but this is a complicated task due to the non linearity and the fact that the distribution of the random variables involved are not always known. To overcome this situation, as well as in \cite{DBS}, one has to use a Monte Carlo approximation. See also Remark \ref{remark-mc}.\\

$L^p(\lambda)$ represents the standard Lebesgue space for the measure $\lambda$. For $p\ge 1$ consider the next processes spaces
\begin{align*}
    \ca{S}^p &= \set{Y: \Omega\times [0,T]\to \R : \norm{Y}_{\ca{S}^p}:=\bb{E}\parent{\underset{t\in [0,T]}\sup |Y_t|^p}^{\frac{1}{p}}<\infty }, \\
    L_{W}^p(\Rd) &= \set{Z:\Omega\times [0,T]\to \Rd :\norm{Z}_{W}^p=\bb{E}\parent{\int_0^T|Z_t|^pdt}<\infty}, \\
    L_{\mu}^p(\R) &= \set{U:\Omega\times [0,T]\times\Rd\to \bb{R} :\norm{U}_{\mu}^p=\bb{E}\parent{\int_0^T\int_{\Rd}|U_t(y)|^p\lambda(dy)dt}<\infty}.
\end{align*}
We will only work with $p=2$ or $p=1$ and denote $\ca{B}^2 = \ca{S}^2\times L_{W}^2(\Rd)\times L_{\mu}^2(\R)$. \\

In order to estimate errors we need a solution to compare, the conditions \textbf{(C)} guarantee the existence and uniqueness of a solution $(X,Y,Z,U)\in\ca{S}^2\times\ca{B}^2$ to the FBSDEJ (\ref{eq:fbsdej}) with starting point $x$, and such that (see \cite[Thm 2.1]{pardoux})
\begin{align}\label{eq: sol-bound}
    ||(X,Y,Z,U)||_{\ca{S}^2\times\ca{B}^2}^2 &\le C_2(1+|x|^2), \\
    \bb{E}\parent{\sup_{s\le u\le t}|X_u-X_s|^2} &\le C_2(1+|x|^2)|t-s|, \\
     \bb{E}\parent{\sup_{s\le u\le t}|Y_u-Y_s|^2} &\le C_2\left[ (1+|x|^2)|t-s|^2 + ||Z||^2_{L^2(W;[s,t])} + ||U||^2_{L^2(\mu;[s,t])} \right].
\end{align}
We also introduce the averaged conditional expectations
\begin{equation}\label{barras}
\begin{aligned}
    \overline{Z}_{t_i}=\dfrac{1}{h}\bb{E}_i\parent{\integral Z_t dt},\quad \overline{\Gamma}_{t_i}=&~{} \dfrac{1}{h}\bb{E}_i\parent{\integral\Gamma_td t}.
    \end{aligned}
\end{equation}
An important quantity to define is the $L^2$-regularity of the solutions $(Z,\Gamma)$ (see \cite{bruno} and \cite{DBS}):
\begin{equation}\label{errores0}
\begin{aligned}
    \varepsilon^Z(h) &:= \bb{E}\parent{\sum_{i=0}^{N-1}\integral |Z_t-\overline{Z}_{t_i}|^2 dt},\\    
    \varepsilon^\Gamma(h) &:= \bb{E}\parent{\sum_{i=0}^{N-1}\integral |\Gamma_t-\overline{\Gamma}_{t_i}|^2 dt}.
\end{aligned}
\end{equation}
Both quantities can be made arbitrarily small, see Proposition \ref{chicos} below.

\subsection{Notation}
Along this paper, $C>0$ will denote a fixed constant, only depending on the dimension, but not on a partition. It may change from one line to another. Also, the notation $a\lesssim b$ means that there exists $C>0$ such that $a\leq Cb$, with $C$ independent of the partition.\\

Only in stochastic integral context, for a cadlag process $(P_s)_s$, $\Delta P_s:=P_s - P_{s^-}$ stands for the jump of $P$ at time $s$. From \cite[Sections 2 and 4]{applebaum}, we recall the definition of the stochastic integral with respect to $\mu$. For a process $U\in L_{\mu}^1(\Rd)$,
\begin{align*}
    \int_{s}^t \int_{\Rd} U(r,y)\mu(ds,dy) := \sum_{r\in (s,t]} U(r,\Delta P_r)\ind{\Rd} (\Delta P_r),
\end{align*}
where 
\begin{align*}
    \parent{P_s = \int_{\Rd} x\mu(s,dx)}_s,
\end{align*}
is a compound Poisson process (see \cite[Thm 2.3.10]{applebaum}). And therefore,  
\begin{align*}
    \int_{s}^t \int_{\Rd} U(r,y)\bar{\mu}(ds,dy) = \sum_{r\in (s,t]} U(r,\Delta P_r)\ind{\Rd} (\Delta P_r) - \int_{s}^t \int_{\Rd} U(r,y)\lambda(dy) dr.
\end{align*}
For sake of simplicity and to avoid an overload of parenthesis, for $Y_1,Y_2$ random variables and $Z_1,Z_2\in\Rd$ random vectors, we follow the next convention
\begin{align*}
    \bb{E}|Z_1-Z_2|^2 = \bb{E} \parent{|Z_1-Z_2|^2},\quad \bb{E}(Y_1-Y_2)^2 = \bb{E}\parent{(Y_1-Y_2)^2}.
\end{align*}
For $a,b\in\Rd$, we denote $a\cdot b$ the dot product in $\Rd$. We write $|a|=a\cdot a$ and $|x|=\sqrt{x^2}$ whenever $x\in \R$. We also use the convention
\[
\int f(s)ds = \begin{pmatrix}\vdots\\ \int f_i(s) ds\\ \vdots\end{pmatrix},
\]
whenever $f:\R\to\Rd$.

\section{Preliminaries}\label{prelim}

We are going to define random variables by composing neural networks with random variables. For technical reasons, these new variables are needed in some $L^p(\Omega,\ca{F},\p)$ space, for $p\in[1,+\infty)$.

\begin{lemma}\label{lemma:nnBound}
Let $X\in L^2 (\Omega,\ca{F},\p;\Rd)$, $W\in\bb{R}^{d}$, $b\in\bb{R}$. Define $\theta=(W,b)\in\R^{d+1}$ and $\ca{U}(\cdot;\theta)$ the associated single layer neural network with input dimension $d$ and output dimension one. If the activation function $\sigma:\R\to \R$ is such that $|\sigma(x)|\le (1+|x|)$ for every $x\in\R$, then  $\ca{U}(X;\theta)\in L^2 (\Omega,\ca{F},\p;\R)$.

\begin{proof}
    Without loss of generality we assume a simple NN. Recall the definition of neural networks, let $\theta = (W,\overline{W}, b,\overline{b})\in\Rd\times\R\times\Rd\times\R$ and let $ \ca{U}(X;\theta) := \overline{W}\sigma(W\cdot X + b) + \overline{b}$,
    \begin{align*}
       \bb{E}(\ca{U}(X;\theta)^2) = \bb{E}(|\overline{W}\sigma(W\cdot X + b) + \overline{b}|^2)\lesssim 
       \overline{W}^2\bb{E}(1+|W\cdot X+b|^2) + \overline{b}^2 < +\infty.
    \end{align*}
 This ends the proof.   
\end{proof}
\end{lemma}

\subsection{Useful results from Stochastic Calculus}

\begin{lemma}[Martingale Representation Theorem, \cite{lukasz}]\label{MRT}
For any martingale $M$ there exists $(Z,U)\in L^2_{W}(\Rd)\times L^2_{\mu}(R)$ such that for $t\in[0,T]$
\[
    M_t = M_0 + \int_0^t Z_s dW_s + \int_0^t \int_{\Rd} U(s,y)\overline\mu (ds,dy).
\]
\end{lemma}

We will need the next property involving conditional expectation, Itô isommetry and that $W$ is independent of $\overline{\mu}$. 

\begin{lemma}(Conditional Ito isommetry)\label{lem:Ito}
For $V^1,V^2\in L^2(\mu)$ and $H,K\in L^2_W(\Rd)$,
\begin{align}
    \bb{E}_i\parent{\integral H_rdW_r\integral K_rdW_r} &= \bb{E}_i\parent{\integral H_rK_r dr},\label{eq:property1}\\
    \bb{E}_i\parent{\integral\int_{\Rd}V^1(s,z)\overline{\mu}(ds,dz)\integral\int_{\Rd}V^2(s,z)\overline{\mu}(ds,dz)}&=\bb{E}_i\parent{\integral\int_{\Rd}V^1(s,z)V^2(s,z)\lambda(dz)ds},
    \nonumber\\
    \bb{E}_i\parent{\integral\int_{\Rd}V^1(r,y)\overline{\mu}(dy,dr)\integral H_r dW_r} &= 0. \nonumber
\end{align}
\end{lemma}

\begin{proof}
Follows from the classical Ito isommetry.
\end{proof}

\begin{lemma}(Conditional Fubini)\label{condFubini}
Let $H\in L^2_{\mu}(\Rd)$ be a $\bb{F}$-adapted process and $t>0$, then
\[
    \bb{E}\parent{\int_{\Rd}\int_{t_i}^{t_{i+1}} H(s,y)ds\lambda(dy)\bigg|\ca{F}_{t_i}} = \int_{\Rd}\bb{E}\parent{\int_{t_i}^{t_{i+1}} H(s,y)ds\bigg|\ca{F}_{t_i}}\lambda(dy).
\]
\end{lemma}

\begin{proof}The proof is standard, but we included it by the sake of completeness. Let $A\in\ca{F}_{t_i}$, we have to prove that
    \[
        \int_A \parent{\int_{\Rd} \bb{E}_i\parent{\integral H(s,y)ds}\lambda(dy) }d\p (\omega) =
        \int_{A}\parent{\int_{\Rd}\integral H(s,y)(\omega)ds\lambda(dy)}d\p (\omega).
    \]
    Note that because of $H\in L^2_{\mu}(\Rd)$,
    \begin{align*}
        \int_\Omega \integral \!\! \int_{\Rd} \! |H(s,y)(\omega)|^2  \lambda(dy) ds d\p (\omega) < \infty;
    \end{align*}  
which means that $H$ can be seen as an element of $\in L^2(\Omega\times [t_i,t_{i+1}]\times \Rd)\subset L^1(\Omega\times [t_i,t_{i+1}]\times \Rd)$, both spaces endowed with the correspondent finite product measure. Then we can use classical Fubini theorem:
    \begin{align*}
        \int_{A}\parent{\int_{\Rd}\integral H(s,y)(\omega)ds\lambda(dy)}d\p (\omega) &= \int_{\Rd}\parent{\int_{A} \integral H(s,y)(\omega)dsd\p (\omega)}\lambda(dy)\\
        &=\int_{\Rd}\parent{\int_{A} \bb{E}_i\parent{\integral H(s,y)(\omega)ds}d\p (\omega)}\lambda(dy)\\
        &= \int_{A}\parent{\int_{\Rd} \bb{E}_i\parent{\integral H(s,y)(\omega)ds}\lambda(dy)}d\p (\omega).
    \end{align*}
This finishes the proof.
\end{proof}

Recall that a martingale $(M_t)_t$ a sequence of random variables (i.e., a stochastic process) for which, at a particular time, the conditional expectation of the next value in the sequence, regardless of all prior values, is equal to the present value. 

\begin{Cl}\label{lem:martingales}
If $(M_t)_t$ is a martingale, and $\Delta M_i:=M_{t_{i+1}}-M_{t_i}$, then the mapping $X\mapsto\bb{E}_i(X\Delta M_i)$ vanishes on $L^2(\ca{F}_{t_i})$.

\begin{proof}
Given $X\in L^2(\ca{F}_{t_i})$, by using that this variable is $\ca{F}_{t_i}$-measurable and $\bb{E}(M_{t_{i+1}}|\ca{F}_{t_i})=M_{t_i}$,
\begin{align*}
    \bb{E}(X\Delta M_i|\ca{F}_{t_i})=X\bb{E}(\Delta M_i|F_{t_i})=0,
\end{align*}
as required.
\end{proof}

\end{Cl}

\subsection{Measuring the error}
Let $i\in \{0,\ldots, N-1\}$, as stated in Subsection \ref{Discretization}. We follow the procedure taken in \cite{DBS}, with key modifications. Let us use the ideas of \cite{bruno} to define $\ca{F}$-adapted discrete processes
\begin{align}
    \widehat{\ca{V}}_{t_i}&=\bb{E}_i\left(\widehat{\ca{U}}_{i+1}(X^\pi_{t_{i+1}})\right) + f\left( t_i,X^\pi_{t_i},\widehat{\ca{V}}_{t_i}, \gb{Z}_{t_i},\gb{\Gamma}_{t_i}\right)h,
    \label{eq:v}
    \\
    \gb{Z}_{t_i}&= \dfrac1{h}\bb{E}_i \left( \widehat{\ca{U}}_{i+1}(X^\pi_{t_{i+1}})\Delta W_{i} \right),\\
    \gb{\Gamma}_{t_i}&= \dfrac1{h} \bb{E}_i\left(\widehat{\ca{U}}_{i+1}(X^\pi_{t_{i+1}})\Delta M_{i}\right),
\end{align}
where $\widehat{\ca{V}}_{t_i}$ is well-defined for sufficiently small $h$ by Lemma \ref{lemma:fixedPoint} and the variables $\gb{Z}_{t_i}$, $\gb{\Gamma}_{t_i}$ are defined below.


\begin{lemma}\label{lemma:fixedPoint}
    The process $\widehat{\ca{V}}_{t_i}$ is well-defined.
    \begin{proof}
        Let $i\in\set{0,...,N-1}$ and $g:L^2\to L^2$ be defined as
        \begin{align*}
            g(Y)=\bb{E}_i\left(\widehat{\ca{U}}_{i+1}(X^\pi_{t_{i+1}})\right) + f\left( t_i,X^\pi_{t_i},Y, \gb{Z}_{t_i},\gb{\Gamma}_{t_i}\right)h. 
        \end{align*}
        This function is well-defined by the properties of $f$ and the Lemma \ref{lemma:nnBound}. Let $Y,\overline Y\in L^2$, then $\p$ a.s $g(Y)-g(\overline Y) \le h |Y-\overline Y|$, therefore
        \begin{align*}
            \norm{g(Y)-g(\overline  Y)}_{L^2} \le h ^2\norm{Y-\overline Y}_{L^2}
        \end{align*}
        Taking sufficiently small $h$ we can see that this function is a contraction on $L^2$, and therefore, by applying Banach's fixed point theorem, we conclude the proof.
    \end{proof}
\end{lemma}

For fixed $i\in \set{0,...,N}$, let $N_t$ be a process defined as $N_t := \bb{E}\parent{\ug{i+1}(\xscheme{i+1})\Big|\ca{F}_t}$ for $t\in [t_i,t_{i+1}]$. It is not difficult to see that $N_t$ is a martingale and therefore, by Martingale Representation Theorem (see Lemma \ref{MRT}), there exist $(\widehat{Z},\widehat{U})\in L^2_{\mu}\times L^2_{W}$ such that
\[
    N_t = N_{t_i} + \int_{t_i}^{t}\widehat{Z}_s \cdot dW_s + \int_{t_i}^{t}\int_{\Rd} \widehat{U}_s (y) \overline{\mu}(ds,dy).
\]
By taking $t=t_{i+1}$ and using \eqref{Esperanza_condicional},
\[
    \ug{i+1}(\xscheme{i+1}) = \bb{E}_i\parent{\ug{i+1}(\xscheme{i+1})} + \integral\widehat{Z}_s \cdot dW_s + \integral\int_{\Rd} \widehat{U}_s (y) \overline{\mu}(ds,dy).
\]
By multiplying by $\Delta W_i$ and $\Delta M_i$, then taking $\bb{E}_i$ and using Itô isometry,
\begin{align*}
    \gb{Z}_{t_i} &= \frac{1}{h} \bb{E}_i\parent{\integral\widehat{Z}_sds},\\
    \gb{\Gamma}_{t_i} &= \frac{1}{h} \bb{E}_i\parent{\integral\int_{\Rd}\widehat{U}_s(y)\lambda(dy)ds}.
\end{align*}
Let 
\begin{equation}\label{Ugorrobarra}
 \gb{U}_{t_i}(y):=\frac1h\bb{E}_i\parent{\integral \widehat{U}_s(y)ds}.
\end{equation}
By Lemma \ref{condFubini} one can see that
\begin{equation}\label{gamma_ti}
    \gb{\Gamma}_{t_i} = \frac1h \bb{E}_i\parent{\integral\int_{\Rd}\widehat{U}_s(y)\lambda(dy)ds}  = \int_{\Rd} \gb{U}_{t_i}(y) \lambda(dy).
\end{equation}
The last equality can be seen as an analogous to (\ref{eq:gamma}) and makes sense with the notation $\gb{\Gamma}_{t_i}=\prom{\gb{U}_{t_i}}$. Also, we can establish the following useful bound: 
\[
    \esp{\gb{\Gamma}_{t_i} - \prom{\ca{G}}_i (\xscheme{i};\theta)}^2 \lesssim \bb{E}\parent{\norm{\overline{\widehat{U}}_{t_i}(\cdot) - \ca{G}_i(\xscheme{i},\cdot,\theta)}^2_{L^2(\lambda)} }.
\]
Indeed, from \eqref{gamma_ti} and \eqref{calG}, H\"older and the fact that $\lambda(dy)$ is a finite measure
\[
\begin{aligned}
\esp{\gb{\Gamma}_{t_i} - \prom{\ca{G}}_i (\xscheme{i};\theta)}^2 = &~{} \esp{\int_{\Rd} \gb{U}_{t_i}(y) \lambda(dy) -  \int_{\mathbb R^d}\ca{G}_i(\xscheme{i},y;\theta)\lambda(dy)}^2 \\
\le  &~{} C \bb{E}\parent{\norm{\overline{\widehat{U}}_{t_i}(\cdot) - \ca{G}_i(\xscheme{i},\cdot,\theta)}^2_{L^2(\lambda)} }.
\end{aligned}
\]

We can find deterministic functions $v_i, z_i, \gamma_i$ such that $v_i(\xscheme{i}) = \widehat{\ca{V}}_{t_i}, z_i(\xscheme{i}) = \gb{Z}_{t_i}$ and $\gamma_i(y, \xscheme{i}) = \gb{U}_{t_i}(y)$. With the previous setup, the natural extension of the terms to estimate the error of the scheme shown on \cite{DBS} must be
\begin{equation}\label{errores}
\begin{aligned}
    \ca{E}_i^v &= \inf_{\xi} \bb{E}\barras{v_i(\xscheme{i})-\ca{U}_i(\xscheme{i};\xi)}^2, \quad 
    \ca{E}_i^z = \inf_{\xi} \bb{E}\barras{ z_i(\xscheme{i})-\ca{Z}_i(\xscheme{i};\xi)}^2\\
\ca{E}_i^{\gamma} &= \inf_{\xi} \bb{E}\parent{ \int_{\Rd} \left| \gamma_i(y, \xscheme{i})-\ca{G}_i(\xscheme{i},y;\xi) \right|^2\lambda(dy)}.
\end{aligned}
\end{equation}
The expected values can be written as a integral with respect a probability measure in $\Rd$ and therefore, applying the Theorem \ref{uat}, these quantities can be made arbitrarily small as the dimension of the parameters increases.\\

The following results will be useful in the proof of the main result.

\begin{prop}[\cite{bruno}, Prop. 2.1]\label{lemma:y-bound}
    There exists a constant $C>0$ independent of the step $h$ such that 
    \begin{align*}
        \sum_{i=0}^{N-1}\bb{E}\parent{\integral|Y_s-Y_{t_i}|^2ds} \leq Ch.
    \end{align*}
\end{prop}

We also need the following boundedness result.
\begin{lemma}\label{lemma:f-int-bound}
    Let $\Theta_s = (s,X_s,Y_s,Z_s,\Gamma_s)$ the true solution to (\ref{eq:fbsdej}). Then
    \[
        \bb{E}\parent{\int_0^T |f(\Theta_s)|^2ds} < \infty .
    \]
    \begin{proof}
        By Lipschitz condition on $f$ we have that
        
        \[
            |f(s,X_s,Y_s,Z_s,\Gamma_s)|^2 \lesssim |f(s,0,0,0,0)|^2 + |X_s|^2 + |Y_s|^2 + |Z_s|^2 + |\Gamma_s|^2
        \]
        
        and
        
        \[
            \sup_{s\in I_{T}} |f(s,0,0,0,0)|^2 < \infty
        \]
        
        then, integrating on $[0,T]$ and taking expected value,
        \[
            \bb{E}\parent{\int_0^T |f(s,X_s,Y_s,Z_s,\Gamma_s)|^2}\lesssim T + ||(X,Y,Z,U)||_{\ca{S}^2\times\ca{B}^2}^2 <\infty
        \]
    where we used that the initial value $x\in \Rd$ that appears on equation (\ref{eq: sol-bound}) is fixed.
    \end{proof}
\end{lemma}

For the next result we use the additional assumption stated in \eqref{CC}: for each $y\in\Rd$, the map $x\to\beta(x,y)$ admits a Jacobian matrix $\nabla\beta(x, y)$ such that the function $a(x,\xi; y):=\xi^T (\nabla\beta(x, y) + I)\xi$ satisfies $a(x,\xi;y)\ge |\xi|^2 K^{-1}$ or $a(x,\xi;y)\le -|\xi|^2 K^{-1}$, uniformly on $x,y\in \Rd$. As usual, big $O(h)$ notation means the existence of a fixed constant $C>0$ independent of small $h$ such that the quantity is bounded by $Ch.$
\begin{prop}\label{chicos}
    Under the additional assumption in \eqref{CC}, one has
    \begin{align*}
        \varepsilon^\Gamma(h) = O(h) \qquad and \qquad \varepsilon^Z (h) = O(h).
    \end{align*}
\end{prop}
\begin{proof}
        See \cite{bruno} for the first one and \cite{zha04} for the second one.
\end{proof}

\section{Proof of the Main Theorem}\label{sect:Proof}


As stated previously, the proof of our main result, Theorem \ref{MT1}, is deeply inspired in the the case without jumps considered in \cite{DBS}. We follow the lines of that proof, with some important differences because of the nonlocal character of our problem.

\subsection{Step 1} 
Recall $\widehat{\ca{V}}_{t_i}$ introduced in \eqref{eq:v}. The purpose of this part is to obtain a suitable bound of the term  $\esp{Y_{t_i}-\widehat{\ca{V}}_{t_i}}^2$ in terms of more tractable terms. We have

\begin{lemma}
There exists $C>0$ fixed such that for any $0<h<1$ sufficiently small, one has
\begin{align}\label{step1}
    \esp{Y_{t_i}-\widehat{\ca{V}}_{t_i}}^2 \le &~{} Ch^2+C\bb{E}\parent{\integral|Y_s-Y_{t_i}|^2ds}+ C\bb{E}\parent{\integral|Z_s-\overline{Z}_{t_i}|^2ds} \nonumber \\
    &~{} +C\bb{E}\parent{\integral|\Gamma_s-\overline{\Gamma}_{t_i}|^2ds} +Ch \bb{E}\parent{\integral f(\Theta_r)^2dr}  \nonumber  \\
    &~{} + C(1+Ch)\bb{E} \left| Y_{t_{i+1}}-\ug{i+1} (X^\pi_{t_{i+1}}) \right|^2,
\end{align}
with $\Theta_r=(r,X_r,Y_r,Z_r,\Gamma_r)$.
\end{lemma}

The rest of this subsection is devoted to the proof of this result. Subtracting the equation \eqref{eq:bsdej} between $t_i$ and $t_{i+1}$, we obtain
\begin{align}
  \Delta Y_i =  Y_{t_{i+1}}-Y_{t_i}=-\int_{t_i}^{t_{i+1}} f(\Theta_s)ds +\int_{t_i}^{t_{i+1}}Z_s\cdot dW_s+\int_{t_i}^{t_{i+1}}\int_{\Rd} U_s(y)\overline{\mu}(ds,dy).
    \label{eq:resta}
\end{align}
Using the definition of $\widehat{\ca{V}}_{t_i}$ in \eqref{eq:v},
\[
\begin{aligned}
    Y_{t_i}-\widehat{\ca{V}}_{t_i}=&~{}  Y_{t_{i+1}} -\Delta Y_i -\widehat{\ca{V}}_{t_i}\\
    =&~{} Y_{t_{i+1}}+\integral [f(\Theta_s)-f(\widehat{\Theta}_{t_i})]ds-\integral Z_s\cdot dW_s - \integral \int_{\Rd} U_s(y)\overline{\mu}(ds,dy)\\
    &~{} -\bb{E}_i(\ug{i+1}(X^\pi_{t_{i+1}})).
\end{aligned}
\]
Here $\widehat{\Theta}_{t_i}=(t_i,X^\pi_{t_i},\widehat{\ca{V}}_{t_i}, \gb{Z}_{t_i},\gb{\Gamma}_{t_i})$. Then, by applying the conditional expectation for time $t_i$ given by $\bb{E}_i$ and using that, in this case, the stochastic integrals are martingales
\begin{align*}
    Y_{t_i}-\widehat{\ca{V}}_{t_i}=\bb{E}_i(Y_{t_{i+1}}-\ug{i+1}(X^\pi_{t_{i+1}})) + \bb{E}_i\parent{\integral [f(\Theta_s)-f(\widehat{\Theta}_{t_i})]ds} = a+b.
\end{align*}
Using the classical inequality $(a+b)^2\le (1+\gamma h)a^2+(1+\frac{1}{\gamma h})b^2$ for $\gamma>0$ to be chosen, we get
\begin{equation}\label{parada1}
\begin{aligned}
    \bb{E}\barras{Y_{t_i}-\widehat{\ca{V}}_{t_i}}^2\le &~{} (1+\gamma h)
    \bb{E} \left[ \bb{E}_i\parent{Y_{t_{i+1}}-\ug{i+1}(X^\pi_{t_{i+1}})} \right] ^2  \\
    &~{} + \parent{1+\frac{1}{\gamma h}}\bb{E} \left[\bb{E}_i\parent{\integral [f(\Theta_s)-f(\widehat{\Theta}_{t_i})]ds}\right]^2.
\end{aligned}
\end{equation}
With no lose of generality, because we are looking for bounds, we can replace $[f(\Theta_s)-f(\widehat{\Theta}_{t_i})]$ by $|f(\Theta_s)-f(\widehat{\Theta}_{t_i})|$. Also, we can drop the $\bb{E}_i$ due to the law of total expectation. The Lipschitz condition on $f$ in \eqref{CC} allows us to give a bound in terms of the difference between $\Theta_s$ and $\widehat{\Theta}_{t_i}$. Indeed, for a fixed constant $K>0$,
\begin{align*}
|f(\Theta_s)-f(\widehat{\Theta}_{t_i})|\le K\parent{|s-t_i|^{1/2}+|X_s-X^{\pi}_{t_i}|+|Y_s-\vg{t_i}|+|Z_s-\gb{Z}_{t_i}|+|\Gamma_s-\gb{\Gamma}_{t_i}|}.
\end{align*}
Therefore, we have the bound
\begin{align*}
    \bb{E}\parent{\integral |f(\Theta_s)-f(\widehat{\Theta}_{t_i})|ds}^2 \le &~{} Ch\left[ h^2+\bb{E}\parent{\integral |X_s-X^{\pi}_{t_i}|^2ds} + \bb{E}\parent{\integral |Y_s-\vg{t_i}|^2ds} \right.\\
    &\qquad +\left. \bb{E}\parent{\integral |Z_s-\gb{Z}_{t_i}|^2ds} + \bb{E}\parent{\integral |\Gamma_s-\gb{\Gamma}_{t_i}|^2ds}\right],
\end{align*}
where the Lipschitz constant $K$ was absorbed by $C$. Using now the triangle inequality $|Y_s-\vg{t_i}|^2 \leq 2|Y_s-Y_{t_i}|^2 +2|Y_{t_i}-\vg{t_i}|^2$,  and the approximation error of the $X$ scheme \eqref{error_X}, we find
\begin{align}
   &  \bb{E}\parent{\integral |f(\Theta_s)-f(\widehat{\Theta}_{t_i})|ds}^2 \\
    &\quad \le C h\left[  h^2+ 
    2\bb{E}\parent{\integral |Y_s-Y_{t_i}|^ 2 ds} +2h \esp{Y_{t_i}- \vg{t_i}}^2 \right. \nonumber\\
    &\left. \qquad\qquad
    + \bb{E}\parent{\integral |Z_s-\gb{Z}_{t_i}|^2 ds} + \bb{E}\parent{\integral |\Gamma_s-\gb{\Gamma}_{t_i}|^2ds} \right], \label{eq:4.9}
\end{align}
and therefore, replacing in \eqref{parada1},
\begin{align}
   & \esp{Y_{t_i}-\widehat{\ca{V}}_{t_i}}^2 \nonumber\\
    &\le \left(1+\gamma h \right)\esp{\bb{E}_i\left[Y_{t_{i+1}}-\ug{i+1} (X^\pi_{t_{i+1}}) \right]  }^2 \nonumber\\
    & \quad + \parent{1+\gamma h} \frac{C}{\gamma} \left[ h^2 + 
    \bb{E}\parent{\integral |Y_s-Y_{t_i}|^ 2 ds} +h \esp{Y_{t_i}- \vg{t_i}}^2 \right. \nonumber\\
    &\left. \qquad \qquad  \qquad  \quad \quad
    + \bb{E}\parent{\integral |Z_s-\gb{Z}_{t_i}|^2 ds} + \bb{E}\parent{\integral |\Gamma_s-\gb{\Gamma}_{t_i}|^2ds} \right]. 
\end{align}
Recall $\overline Z_{t_i}$ and $\overline \Gamma_{t_i}$ introduced in \eqref{barras}. Now, we are going to prove the following
\begin{align}
    \bb{E}\parent{\integral |Z_s-\gb{Z}_{t_i}|^2 ds}&=\bb{E}\parent{\integral|Z_s-\overline{Z}_{t_i}|^2 ds}+h \esp{\overline{Z}_{t_i}-\gb{Z}_{t_i}}^2.
    \label{eq:ortogonal1}\\
    \bb{E}\parent{\integral |\Gamma_s-\gb{\Gamma}_{t_i}|^2 ds}&=\bb{E}\parent{\integral|\Gamma_s-\overline{\Gamma}_{t_i}|^2ds}+h \esp{\overline{\Gamma}_{t_i}-\gb{\Gamma}_{t_i}}^2.
    \label{eq:ortogonal2}
\end{align}
Let us prove the latter, the former is analogous. Recall that the $\Gamma$ components represents the nonlocal part and therefore is one dimensional.
\begin{align*}
    |\Gamma_t-\gb{\Gamma}_{t_i}|^2 = | (\Gamma_t-\overline{\Gamma}_{t_i}) + (\overline{\Gamma}_{t_i}-\gb{\Gamma}_{t_i}) |^2 = (\Gamma_t-\overline{\Gamma}_{t_i})^2 + (\overline{\Gamma}_{t_i}-\gb{\Gamma}_{t_i})^2 + 2(\Gamma_t-\overline{\Gamma}_{t_i})(\overline{\Gamma}_{t_i}-\gb{\Gamma}_{t_i}).
\end{align*}
It is sufficient to establish that the double product is $0$ when integrating and taking expectation. Recall that $\overline{\Gamma}_{t_i}$ from \eqref{barras}
is a $\ca{F}_{t_i}$ measurable random variable. Then,
\begin{align*}
    \integral \parent{\Gamma_t-\overline{\Gamma}_{t_i}}(\overline{\Gamma}_{t_i}-\gb{\Gamma}_{t_i})dt& = \parent{\integral (\Gamma_t-\overline{\Gamma}_{t_i})dt}(\overline{\Gamma}_{t_i}-\gb{\Gamma}_{t_i})\\
    &= \left[\integral \Gamma_t dt - \bb{E}_i\parent{\integral\Gamma_t dt}\right] (\overline{\Gamma}_{t_i}-\gb{\Gamma}_{t_i}).
\end{align*}
Due to the $\ca{F}_{t_i}$-measurability of the right side of the last multiplication and the $L^2(\p)$ orthogonality, taking expectation annihilates the last term. Therefore, equations (\ref{eq:ortogonal1}) and (\ref{eq:ortogonal2}) are proven. By multiplying (\ref{eq:resta}) by $\Delta W_i$ and taking $\bb{E}_i$,
\begin{align*}
    \bb{E}_i\parent{\Delta W_i Y_{t_{i+1}}} + \bb{E}_i\parent{\Delta W_i\integral f(\Theta_r)dr} =&~{}  \bb{E}_i\parent{\integral dW_r\integral Z_r dW_r}\\
    &+\bb{E}_i\parent{\integral\int_{\Rd}U_r(y)\overline{\mu}(dy,dr)\integral dW_r}\\
    =&~{} \bb{E}_i\parent{\integral Z_r dr} = h \overline{Z}_{t_i},
\end{align*}
where we have used Lemma \ref{eq:property1}. Then, subtracting $h \gb{Z}_{t_i}=\bb{E}_i(\ug{i+1}(X^{\pi}_{t_{i+1}})\Delta W_i)$,
\begin{align*}
    h (\overline{Z}_{t_i}- \gb{Z}_{t_i})= \bb{E}_i\left[\Delta W_i (Y_{t_{i+1}}-\ug{i+1}(X^{\pi}_{t_{i+1}}))\right] &+ \bb{E}_i\parent{\Delta W_i\integral h(\Theta_r)dr}.
\end{align*}
By multiplying (\ref{eq:resta}) by $\Delta M_i$ and taking $\bb{E}_i$,
\begin{align*}
    & \bb{E}_i\parent{\Delta M_i Y_{t_{i+1}}} + \bb{E}_i\parent{\Delta M_i\integral f(\Theta_r)dr} \\
    &= \bb{E}_i\parent{\integral\int_{\Rd}\overline{\mu}(ds,dy)\integral Z_r\cdot dW_r} + \bb{E}_i\parent{\integral \int_{\Rd} \overline{\mu}(dr,dy)\integral \int_{\Rd}U_r(y) \overline{\mu}(dr,dy)} \\
    &= \bb{E}_i\parent{\integral\int_{\Rd}U_r(y)\lambda(dy)ds} = h \overline{\Gamma}_{t_i}.
\end{align*}
Then, subtracting $h\gb{\Gamma}_{t_i}=\bb{E}_i\parent{\ug{i+1}(X^{\pi}_{t_{i+1}})\Delta M_i}$,
\begin{align*}
    h(\overline{\Gamma}_{t_i}-\gb{\Gamma}_{t_i}) = \bb{E}_i\left[\Delta M_i\parent{Y_{t_{i+1}}-\ug{i+1}(X^{\pi}_{t_{i+1}})}  \right] + \bb{E}_i\parent{\Delta M_i\integral f(\Theta_r)dr}.
\end{align*}
Summarizing, one has
\begin{align*}
    h(\overline{Z}_{t_i}-\gb{Z}_{t_i}) = &~{} \bb{E}_i \left[ \Delta W_i \parent{Y_{t_{i+1}}-\ug{i+1}(X^{\pi}_{t_{i+1}})-\bb{E}_i\left[Y_{t_{i+1}}-\ug{i+1}(X^{\pi}_{t_{i+1}})\right]}\right] \\
    &~{} + \bb{E}_i\left[ \Delta W_i\integral f(\Theta_r)dr \right];\\
    h(\overline{\Gamma}_{t_i}-\gb{\Gamma}_{t_i}) =&~{}  \bb{E}_i \left[ \Delta M_i \parent{Y_{t_{i+1}}-\ug{i+1}(X^{\pi}_{t_{i+1}})-\bb{E}_i\left[Y_{t_{i+1}}-\ug{i+1}(X^{\pi}_{t_{i+1}})\right]}\right] \\
    &~{} + \bb{E}_i\left[ \Delta M_i\integral f(\Theta_r)dr \right].
\end{align*}
For the sake of brevity, define now 
\begin{equation}\label{def_H}
H_{i}:=Y_{t_{i}}-\ug{i} (X^\pi_{t_{i}});
\end{equation}
note that it depends on $i$. By the properties related with Itô isometry, from the previous identities we have
\begin{align}
    \bb{E}\parent{h^2 \left| \overline{Z}_{t_i}-\gb{Z}_{t_i} \right|^2} \le &~{} 2dh \parent{\bb{E}(H_{i+1}^2)-\esp{\bb{E}_i(H_{i+1})}^2} + 2d h^2\bb{E}\left[ \integral f(\Theta_r)^2dr\right];
    \label{eq: 4.11z}
    \\
    \bb{E}\parent{h^2 \left| \overline{\Gamma}_{t_i}-\gb{\Gamma}_{t_i} \right|^2} \le &~{} 2\lambda(\Rd) h \parent{\bb{E}(H_{i+1}^2)-\esp{\bb{E}_i(H_{i+1})}^2}  + 2\lambda(\Rd) h^2\bb{E}\left[ \integral f(\Theta_r)^2dr\right].
    \label{eq: 4.11g}
\end{align}
\begin{remark}
Note that in the previous bound is important the finiteness of the Levy measure $\lambda$. The case of more general integro-differential operators, such as the fractional Laplacian mentioned in the introduction, it is an interesting open problem. 
\end{remark}

Let us work with equation (\ref{eq:4.9}). Using (\ref{eq:ortogonal1}) and (\ref{eq:ortogonal2}), 
\begin{align*}
    \esp{Y_{t_i}-\widehat{\ca{V}}_{t_i}}^2
    \le &~{} \left(1+\gamma h\right)\esp{\bb{E}_i(H_{i+1})  }^2 \\
    &~{} + \parent{1+\gamma h} \frac{C}{\gamma} \left[  h^2+\bb{E}\parent{\integral |Y_s-Y_{t_i}|^2 ds} \right. 
    \left.+h \bb{E}|Y_{t_i}- \vg{t_i}|^2\right.\nonumber\\
    &~{} \qquad \qquad \qquad \qquad  + \left.\bb{E}\parent{\integral|Z_s-\overline{Z}_{t_i}|^2ds}+h \esp{\overline{Z}_{t_i}-\gb{Z}_{t_i}}^2  \right. \nonumber\\
    &~{} \qquad \qquad \qquad \qquad \left.
    +  \bb{E}\parent{\integral|\Gamma_s-\overline{\Gamma}_{t_i}|^2ds}+h \esp{\overline{\Gamma}_{t_i}-\gb{\Gamma}_{t_i}}^2 \right].
\end{align*}
Now use (\ref{eq: 4.11z}) and (\ref{eq: 4.11g}) to find that
\begin{align*}
  &~{}  \esp{Y_{t_i}-\widehat{\ca{V}}_{t_i}}^2\\
   &~{} \le  (1+\gamma h)\esp{\bb{E}_i(H_{i+1})}^2 \\
    &~{} \qquad  + (1+\gamma h)\frac{C}{\gamma}\left[ h^2 + \bb{E}\parent{\integral|Y_s-Y_{t_i}|^2ds}\right.+ h\bb{E}\left|Y_{t_i}-\widehat{\ca{V}}_{t_i}\right|^2\\
    &~{} \qquad  \qquad \qquad \qquad  \qquad+\bb{E}\parent{\integral|Z_s-\overline{Z}_{t_i}|^2ds} +\bb{E}\parent{\integral|\Gamma_s-\overline{\Gamma}_{t_i}|^2ds}\\
    &~{} \qquad  \qquad \qquad \qquad \qquad +2d\left[\bb{E}\parent{H_{i+1}^2}-\esp{\bb{E}_i(H_{i+1})}^2\right] + 2dh \bb{E}\parent{\integral f(\Theta_r)^2 dr} \\
   &~{} \qquad  \qquad \qquad \qquad \qquad + 2\lambda(\Rd)\left[\bb{E}\parent{H_{i+1}^2}-\esp{\bb{E}_i(H_{i+1})}^2\right] \left.+2\lambda(\Rd)h \bb{E}\parent{\integral f(\Theta_r)^2 dr} \right].
\end{align*}
Let $\gamma=C(\lambda(\Rd) + d)$ and define $D:=(1+\gamma h)\frac{C}{\gamma}$, then the above term is bounded by
\begin{align*}
    & (1+\gamma h)\esp{\bb{E}_i(H_{i+1})}^2+ Dh^2 + D\bb{E}\parent{\integral|Y_s-Y_{t_i}|^2} + Dh \bb{E}|Y_{t_i}-\widehat{\ca{V}}_{t_i}|^2 +   D \bb{E}\parent{\integral|Z-\overline{Z}_{t_i}|^2ds}\\
    &\qquad+ D\bb{E}\parent{\integral|\Gamma_s-\overline{\Gamma}_{t_i}|^2ds} + (1+\gamma h)\frac{C}{\gamma} 2d\bb{E}\parent{H_{i+1}^2} + 2dDh\bb{E}\parent{\integral f(\Theta_r)^2dr}\\
    &\qquad+ (1+\gamma h)\frac{C}{\gamma} 2\lambda(\Rd)\bb{E}\parent{H_{i+1}^2} + 2\lambda(\Rd) Dh\bb{E}\parent{\integral f(\Theta_r)^2dr} - 2(1+\gamma h) \esp{\bb{E}_i(H_{i+1})}^2
\end{align*}
Note that the first and last term in the last expression are similar, therefore can be subtracted which yields a negative number that can be bounded from above by $0$. Also, we have the similar terms on $\bb{E}\parent{H_{i+1}^2}$ and the integral of $f$ that we put together and bound respectively. Due to the definition of $D$, from now on the constant $C$ has a linear dependence on the dimension $d$ such that $D\le C$.
%
By replacing the last calculation and putting $\esp{Y_{t_i}-\widehat{\ca{V}}_{t_i}}^2$ on the left side
\begin{align*}
   &~{}  (1-Ch)\esp{Y_{t_i}-\widehat{\ca{V}}_{t_i}}^2\\
   &~{}\qquad  \le Ch^2+C\bb{E}\parent{\integral|Y_s-Y_{t_i}|^2ds} + C\bb{E}\parent{\integral|Z_s-\overline{Z}_{t_i}|^2ds}\\
    &~{} \qquad  +C\bb{E}\parent{\integral|\Gamma_s-\overline{\Gamma}_{t_i}|^2ds} + C(1+Ch)\bb{E}\parent{H_{i+1}^2} + Ch \bb{E}\parent{\integral f(\Theta_r)^2dr}.
\end{align*}
Now we have to take $h$ small such that, for example, $Ch\le\frac{1}{2}$ and then
\begin{align*}
    \esp{Y_{t_i}-\widehat{\ca{V}}_{t_i}}^2 \le &~{} Ch^2+C\bb{E}\parent{\integral|Y_s-Y_{t_i}|^2ds}+ C\bb{E}\parent{\integral|Z_s-\overline{Z}_{t_i}|^2ds}\\
    &~{} +C\bb{E}\parent{\integral|\Gamma_s-\overline{\Gamma}_{t_i}|^2ds}+Ch \bb{E}\parent{\integral f(\Theta_r)^2dr}+ C(1+Ch)\bb{E}\parent{H_{i+1}^2} .
\end{align*}
Finally, by recalling that $H_{i+1} =Y_{t_{i+1}}-\ug{i+1} (X^\pi_{t_{i+1}})$, we have established \eqref{step1}.

\subsection{Step 2}

The last term in \eqref{step1},
\[
C(1+Ch)\bb{E} \left| Y_{t_{i+1}}-\ug{i+1} (X^\pi_{t_{i+1}}) \right|^2,
\]
was left without a control in previous step. Here in what follows we provide a control on this term. Recall the error terms $\varepsilon^{Z}(h)$ and $\varepsilon^{\Gamma}(h)$ introduced in \eqref{errores0}. The purpose of this section is to show the following estimate:
\begin{lemma}\label{Lem2p4}
There exists a constant $C>0$ (linearly depending on the dimension $d$) such that,
 \begin{align}
    \max_{i\in\set{0,...,N-1}}\esp{Y_{t_i}-\ug{i}(X_{t_i}^\pi)}^2\le &~{} C\Bigg[N \sum_{i=0}^{N-1} \esp{\ug{i}(X_{t_i}^\pi)-\widehat{\ca{V}}_{t_i}}^2  + h + \varepsilon^{Z}(h)+\varepsilon^{\Gamma}(h) +  \esp{ g(X_T)-g(X_T^\pi) }^2  \Bigg]. \label{step2}
   \end{align}
\end{lemma}
The rest of this section is devoted to the proof of this result. 

\subsubsection{Proof of Lemma \ref{Lem2p4}} We have that $(a+b)^2\ge(1-h)a^2+(1-\frac{1}{h})b^2$ and
\begin{align}
    \label{eq:4.13}
    \esp{Y_{t_i}-\widehat{\ca{V}}_{t_i}}^2&=\esp{\parent{Y_{t_i}-\ug{i}(X^\pi_{t_i})} + \parent{\ug{i}(X^\pi_{t_i})-\widehat{\ca{V}}_{t_i}}}^2\\
    &\ge (1-h)\esp{Y_{t_i}-\ug{i}(X^\pi_{t_i})}^2 + \parent{1-\frac{1}{h}}\esp{\ug{i}(X^\pi_{t_i})-\widehat{\ca{V}}_{t_i}}^2.
    \nonumber
\end{align}
Therefore, we have an upper \eqref{step1} and lower bound for $\esp{Y_{t_i}-\widehat{\ca{V}}_{t_i}}^2$. By connecting these bounds, 
\begin{align*}
    &(1-h)\esp{Y_{t_i}-\ug{i}(X^\pi_{t_i})}^2 + \parent{1-\frac{1}{h}}\esp{\ug{i}(X^\pi_{t_i})-\widehat{\ca{V}}_{t_i}}^2 \\
    & \quad \le Ch^2+C\bb{E}\parent{\integral|Y_s-Y_{t_i}|^2ds} + C\bb{E}\parent{\integral|Z_s-\overline{Z}_{t_i}|^2ds}\\
    &\quad \quad +C\bb{E}\parent{\integral|\Gamma_s-\overline{\Gamma}_{t_i}|^2ds}+Ch \bb{E}\parent{\integral f(\Theta_r)^2dr}+ C(1+Ch)\bb{E}\parent{H_{i+1}^2}.
\end{align*}
Using that for sufficiently small $h$ we have $(1-h)^{-1}\le 2$, we get,
\begin{align*}
&~{} \esp{Y_{t_i}-\ug{i} (\xscheme{i})}^2 \\
 &~{} \leq CN\esp{\ug{i}(\xscheme{i})-\widehat{\ca{V}}_{t_i}}^2 + Ch^2 \\
&~{} \quad + C\left[ \bb{E}\parent{\integral|Y_s-Y_{t_i}|^2ds} + \bb{E}\parent{\integral|Z_s-\overline{Z}_{t_i}|^2ds}+\bb{E}\parent{\integral|\Gamma_s-\overline{\Gamma}_{t_i}|^2ds}\right]  \\
    &~{} \quad +   Ch\bb{E}\parent{\integral|f(\Theta_s)|^2ds}+C\esp{Y_{t_{i+1}}-\ug{i+1}(X^{\pi}_{t_{i+1}}) }^2.
\end{align*}
Notice that the expression on time $t_i$ that we want to estimate, appears on the right side on time $t_{i+1}$, we can iterate the bound and get that $\forall$ $i\in\set{0,...,N-1}$
\begin{align*}
    & \esp{Y_{t_i}-\ug{i} (X_{t_i}^\pi)}^2 \\
     &~{} \le N C\sum_{k=i}^{N-1}\esp{\ug{k}  (X_{t_k}^\pi) -\widehat{\ca{V}}_{t_k}}^2 + C(N-i)h^2 \\
    &~{} \quad+ C\sum_{k=i}^{N-1}\left[ \bb{E}\parent{\int_{t_k}^{t_{k+1}} |Y_s-Y_{t_k}|^2ds} +\bb{E}\parent{\int_{t_k}^{t_{k+1}}  |Z_s-\overline{Z}_{t_k}|^2ds}
    + \bb{E}\parent{\int_{t_k}^{t_{k+1}} |\Gamma_s-\overline{\Gamma}_{t_k}|^2ds}\right] \\
    &\quad+Ch\sum_{k=i}^{N-1}\bb{E}\parent{\int_{t_k}^{t_{k+1}} |f(\Theta_s)|^2ds} +C\esp{Y_{t_{N}}-g(X^{\pi}_{t_N})}^2\\
    &~{} \leq NC\sum_{k=0}^{N-1}\esp{\ug{k} (X_{t_k}^\pi) -\widehat{\ca{V}}_{t_k}}^2 + CNh^2 \\
    &~{} \quad + C\sum_{k=0}^{N-1}\left[ \bb{E}\parent{\int_{t_k}^{t_{k+1}} |Y_s-Y_{t_k}|^2ds} \right. +\bb{E}\parent{\int_{t_k}^{t_{k+1}} |Z_s-\overline{Z}_{t_k}|^2ds}
    + \left.\bb{E}\parent{\int_{t_k}^{t_{k+1}} |\Gamma_s-\overline{\Gamma}_{t_k}|^2ds}\right] \\
    &~{} \quad + Ch\sum_{k=0}^{N-1}\bb{E}\parent{\int_{t_k}^{t_{k+1}}| f(\Theta_s)|^2ds} +C\esp{Y_{t_{N}}-g(X^{\pi}_{t_N})}^2.
\end{align*}
 Applying maximum on $i\in\set{0,...,N-1}$, recalling \eqref{errores0} and the bounds from Lemmas (\ref{lemma:f-int-bound}) and (\ref{lemma:y-bound}), 
   \begin{align*}
    \max_{i\in\set{0,...,N-1}}\esp{Y_{t_i}-\ug{i}(X_{t_i}^\pi)}^2\le &~{} C\Bigg[N\sum_{i=0}^{N-1} \esp{\ug{i}(X_{t_i}^\pi)-\widehat{\ca{V}}_{t_i}}^2  + O(h) +\varepsilon^{Z}(h)+\varepsilon^{\Gamma}(h) +  \esp{g(X_T)-g(X_T^\pi)}^2  \Bigg].
   \end{align*}
   This is nothing that \eqref{step2}.
   
\begin{remark}\label{delta-remark}
	The classic bound used at the beginning of step $2$ could have been stated using a fixed parameter $\delta\in (0,1)$ in the form: $(a+b)^2 \ge (1-h^{\delta})a^2 + (1-\frac{1}{h^\delta})b^2$. This change makes $N$ become $N^{\delta}$, which is better. However, at some point of the proof the value $\delta = 1$ is necessary.
\end{remark}
\subsection{Step 3}

Estimate \eqref{step2} contains some uncontrolled terms on its RHS. Here the purpose is to bound the term
\[
\sum_{i=0}^{N-1} \esp{\ug{i}(X_{t_i}^\pi)-\widehat{\ca{V}}_{t_i}}^2,
\]
in terms of more tractable terms. In this step we will prove
\begin{lemma}\label{lemma:5.3}
There exists $C>0$ fixed, linearly depending on the dimension, such that,
\begin{align}
    \max_{i\in\set{0,...,N-1}}\esp{Y_{t_i}-\ug{i}(X_{t_i}^\pi)}^2\le &~{} C\Bigg[ h + \sum_{i=0}^{N-1} (N\ca{E}_i^v + \ca{E}_i^z + \ca{E}_i^\gamma)  +\varepsilon^{Z}(h)+\varepsilon^{\Gamma}(h) +  \esp{g(X_T)-g(X_T^\pi)}^2  \Bigg],
    \label{step34}
\end{align}
with $\ca{E}_i^v$, $\ca{E}_i^z$ and $\ca{E}_i^\gamma$ defined in \eqref{errores}.
\end{lemma}

In what follows, we will prove \ref{step34}.

\medskip

Fix $i\in\set{0,...,N-1}$. Recall the martingale $(N_t)_{t\in[t_i,t_{i+1}]}$ and take $t=t_{i+1}$,
\begin{align*}
    \ug{i+1} (X^{\pi}_{t_{i+1}}) = \bb{E}_i \parent{\ug{i+1}(X^{\pi}_{t_{i+1}})} + \int_{t_i}^{t_{i+1}} \widehat{Z}_s\cdot dW_s + \int_{t_i}^{t_{i+1}}\int_{\Rd} \widehat{U}_s(y)\overline{\mu}(ds,dy).
\end{align*}
Now we replace the definition of $\widehat{\ca{V}}_{t_i}$,
\begin{align}
    \ug{i+1} (X^{\pi}_{t_{i+1}}) = \widehat{\ca{V}}_{t_i} - f(t_i,X^{\pi}_{t_i},\widehat{\ca{V}}_{t_i}, \gb{Z}_{t_i},\gb{\Gamma}_{t_i})h +\int_{t_i}^{t_{i+1}} \widehat{Z}_s\cdot dW_s + \int_{t_i}^{t_{i+1}}\int_{\Rd} \widehat{U}_s(y)\overline{\mu}(ds,dy).
    \label{eq:rep}
\end{align}

In what follows recall the value of $F$ in the loss function $L_i(\theta)$ (\ref{eq:loss}) evaluated at the point 
\[
(t_i,X^\pi_{t_i},\ca{U}_i(X^\pi_{t_i};\theta) ,\ca{Z}_i(X^\pi_{t_i};\theta),\ca{G}_i(X^\pi_{t_i},\cdot;\theta),h,\Delta W_i),
\] 
and that $\prom{\ca{G}}_i(X_{t_i};\theta)$ is given in \eqref{calG}: 
\begin{align*}
& F\left( t_i,X^\pi_{t_i},\ca{U}_i(X^\pi_{t_i};\theta) ,\ca{Z}_i(X^\pi_{t_i};\theta),\ca{G}_i(X^\pi_{t_i},\cdot;\theta),h,\Delta W_i \right) \\
 &~{} = \ca{U}_i(X^\pi_{t_i};\theta) -hf(t_i,X^\pi_{t_i},\ca{U}_i(X^\pi_{t_i};\theta),\ca{Z}_i(X^\pi_{t_i};\theta) , \prom{\ca{G}}_i(X_{t_i};\theta)) +  \ca{Z}_i(X^\pi_{t_i};\theta)\cdot \Delta W_i +\int_{\Rd} \ca{G}_i(X^\pi_{t_i},y;\theta) \overline{\mu}\big((t_i,t_{i+1}],dy\big).
\end{align*}

Now fix a parameter $\theta$ and replace (\ref{eq:rep}) on $L_i(\theta)$:
\begin{align*}
    &\bb{E} \Big|\ug{i+1} (X^{\pi}_{t_{i+1}})  -F(t_i,X^\pi_{t_i},\ca{U}_i(X^\pi_{t_i};\theta),\ca{Z}_i(X^\pi_{t_i};\theta),\ca{G}_i(X^\pi_{t_i},\cdot;\theta),\Delta t_i,\Delta W_i)\Big|^2\\
    &=\bb{E}\Big|\widehat{\ca{V}}_{t_i} - f(t_i,X^{\pi}_{t_i},\widehat{\ca{V}}_{t_i}, \gb{Z}_{t_i},\gb{\Gamma}_{t_i})h +\int_{t_i}^{t_{i+1}} \widehat{Z}_s\cdot dW_s + \int_{t_i}^{t_{i+1}}\int_{\Rd} \widehat{U}_s(y)\overline{\mu}(ds,dy) - \ca{U}_i(X^{\pi}_{t_i};\theta)\\
    &\quad +h f(t_i,X^{\pi}_{t_i},\ca{U}_i(X_{t_i};\theta),\ca{Z}_i(X_{t_i};\theta), \prom{\ca{G}}_i(X_{t_i};\theta)) -\ca{Z}_i(X^{\pi}_{t_i};\theta)\cdot \Delta W_i-\int_{\Rd}\ca{G}_i(X^{\pi}_{t_i},y;\theta)\overline{\mu}(\Delta t_i,dy)
    \Big|^2\\
    &= \bb{E} \Bigg|\left[ \widehat{\ca{V}}_{t_i}-\ca{U}_i(X^{\pi}_{t_i};\theta) + h\left( f(t_i,X^{\pi}_{t_i},\ca{U}_i(X^{\pi}_{t_i};\theta),\ca{Z}_{i}(X^{\pi}_{t_i};\theta),\prom{\ca{G}}_{i}(X^{\pi}_{t_i};\theta)) - f(t_i,X^{\pi}_{t_i},\widehat{\ca{V}}_{t_i},\gb{Z}_{t_i},\gb{\Gamma}_{t_i}) \right) \right]\\
    &\qquad +
    \left[
    \integral\widehat{Z}_s\cdot dW_s - \integral\ca{Z}_i(X^{\pi}_{t_i};\theta)\cdot dW_s + \integral\int_{\Rd}\widehat{U}(s,y)\overline{\mu}(ds,dy)-\integral\int_{\Rd}\ca{G}_i(X^{\pi}_{t_i},y;\theta)\overline{\mu}(ds,dy)
    \right]
    \Bigg|^2\\
    &=\esp{a+b}^2.
\end{align*}
Note that $b$ is a sum of martingale's differences and therefore $\bb{E}_i (b)=0$. By independence of $\mu$ with $W$, we can deduce that
\begin{align*}
    \bb{E}(b^2) = \bb{E}\parent{\integral[\widehat{Z}_s-\ca{Z}_{i}(X^{\pi}_{t_i};\theta)]dW_s}^2+\bb{E}\parent{\integral\int_{\Rd}[\widehat{U}(s,y)-\ca{G}_i(X^{\pi}_{t_i},y;\theta)]\overline{\mu}(ds,dy)}^2;
\end{align*}
and, since the random variables that appears on $a$ are $\ca{F}_{t_i}$-measurable, $\bb{E}(ab)=\bb{E}\parent{\bb{E}_i(ab)}=\bb{E}\parent{a\bb{E}_i(b)}=0$, we have that
\begin{align*}
    L_i(\theta) &= \bb{E} \left( \widehat{\ca{V}}_{t_i}-\ca{U}_i(X^{\pi}_{t_i};\theta) + h\left[ f(t_i,X^{\pi}_{t_i},\ca{U}_i(X^{\pi}_{t_i};\theta),\ca{Z}_{i}(X^{\pi}_{t_i};\theta),\prom{\ca{G}}_{i}(X^{\pi}_{t_i};\theta)) - f(t_i,X^{\pi}_{t_i},\widehat{\ca{V}}_{t_i},\gb{Z}_{t_i},\gb{\Gamma}_{t_i}) \right] \right)^2\\
    &+ \underbrace{\bb{E}\parent{\integral[\widehat{Z}_s-\ca{Z}_{i}(X^{\pi}_{t_i};\theta)]dW_s}^2+\bb{E}\parent{\integral\int_{\Rd}[\widehat{U}(s,y)-\ca{G}_i(X^{\pi}_{t_i},y;\theta)]\overline{\mu}(ds,dy)}^2}_{c_0}.
\end{align*}
By the same arguments on equations (\ref{eq:ortogonal1}) and (\ref{eq:ortogonal2}),
\begin{align*}
    c_0 =&~{} \bb{E}\parent{\integral|\widehat{Z}_s-\gb{Z}_{t_i}|^2ds} + h\esp{\gb{Z}_{t_i}-\ca{Z}_{i}(X^{\pi}_{t_i};\theta)}^2\\
    &~{} + \bb{E}\parent{\integral\int_{\Rd}|\widehat{U}_s (y)-\gb{U}_{t_i}(y)|^2\lambda(dy)ds} + h\bb{E}\parent{\int_{\Rd} \big(\gb{U}_{t_i}(y)-\ca{G}_{i}(X^{\pi}_{t_i},y;\theta)\big)^2\lambda(dy)}.
\end{align*}
With this decomposition of $L_i(\theta)$, for optimization reasons, we can ignore the part that does not depend on the optimization parameter $\theta$. Let
\begin{align*}
&~{}    \hat{L}_i(\theta)  \\
    &~{} = \bb{E} \left( \widehat{\ca{V}}_{t_i}-\ca{U}_i(X^{\pi}_{t_i};\theta) + h\left[ f(t_i,X^{\pi}_{t_i},\ca{U}_i(X^{\pi}_{t_i};\theta),\ca{Z}_{i}(X^{\pi}_{t_i};\theta),\prom{\ca{G}}_{i}(X^{\pi}_{t_i};\theta)) - f(t_i,X^{\pi}_{t_i},\widehat{\ca{V}}_{t_i},\gb{Z}_{t_i},\gb{\Gamma}_{t_i})\right] \right)^2\\
    &\quad + h\, \esp{\gb{Z}_{t_i}-\ca{Z}_{i}(X^{\pi}_{t_i};\theta)}^2 + h\,\bb{E}\parent{\int_{\Rd} \big(\gb{U}_{t_i}(y)-\ca{G}_{i}(X^{\pi}_{t_i},y;\theta)\big)^2\lambda(dy)}.
\end{align*}
Let $\gamma>0$ and use Young inequality and the Lipschitz condition on $f$ to find that
\begin{align*}
    &\bb{E} \left( \widehat{\ca{V}}_{t_i}-\ca{U}_i(X^{\pi}_{t_i};\theta) + h\left[ f(t_i,X^{\pi}_{t_i},\ca{U}_i(X^{\pi}_{t_i};\theta),\ca{Z}_{i}(X^{\pi}_{t_i};\theta),\prom{\ca{G}}_{i}(X^{\pi}_{t_i};\theta))- f(t_i,X^{\pi}_{t_i},\widehat{\ca{V}}_{t_i},\gb{Z}_{t_i},\gb{\Gamma}_{t_i}) \right] \right)^2\\
    &\le \parent{1+\gamma h}\esp{\widehat{\ca{V}}_{t_i}-\ca{U}_i(X^{\pi}_{t_i};\theta)}^2\\
    &\quad +  \parent{1+\frac{1}{\gamma h}}h^2K^2\bb{E}\parent{|\widehat{\ca{V}}_{t_i}-\ca{U}_i(X^{\pi}_{t_i};\theta)|^2+|\ca{Z}_i (X^{\pi}_{t_i};\theta)-\gb{Z}_{t_i}|^2+|\prom{\ca{G}}_i(\xscheme{i};\theta)-\gb{\Gamma}_{t_i}|^2}\\
    &\le(1+Ch)\esp{\widehat{\ca{V}}_{t_i}-\ca{U}_i(X^{\pi}_{t_i};\theta)}^2 + Ch\Bigg[\bb{E}|\ca{Z}_i (X^{\pi}_{t_i};\theta)-\gb{Z}_{t_i}|^2+\bb{E}\parent{\norm{\gb{U}_{t_i}(\cdot)-\ca{G}_i(\xscheme{i},\cdot;\theta)}_{L^2(\lambda)}^2 }\Bigg].
\end{align*}
Therefore, we have an upper bound on $L(\theta)$ for all $\theta$
\begin{align*}
    \hat{L}(\theta)\le C\esp{\widehat{\ca{V}}_{t_i}-\ca{U}_i(X^{\pi}_{t_i};\theta)}^2 + h\parent{\bb{E}|\ca{Z}_i (X^{\pi}_{t_i};\theta) - \gb{Z}_{t_i}|^2}+h\bb{E}\parent{\norm{\gb{U}_{t_i}(\cdot)-\ca{G}_i(\xscheme{i},\cdot;\theta)}_{L^2(\lambda)}^2}.
\end{align*}
To find a lower bound, we use $(a+b)^2\ge(1-\gamma h)a^2+\left(1-\frac{1}{\gamma h}\right)b^2$ with $\gamma>0$
\begin{align*}
    &\bb{E} \left( \widehat{\ca{V}}_{t_i}-\ca{U}_i(X^{\pi}_{t_i};\theta) + h\left[ f(t_i,X^{\pi}_{t_i},\widehat{\ca{V}}_{t_i},\gb{Z}_{t_i},\gb{\Gamma}_{t_i}) - f(t_i,X^{\pi}_{t_i},\ca{U}_i(X^{\pi}_{t_i};\theta),\ca{Z}_{i}(X^{\pi}_{t_i};\theta),\prom{\ca{G}}_{i}(X^{\pi}_{t_i};\theta)) \right] \right)^2\\
    &\ge \parent{1-Ch}\esp{\widehat{\ca{V}}_{t_i}-\ca{U}_i(X^{\pi}_{t_i};\theta)}^2 -\frac{h}{2} \parent{\bb{E}|\ca{Z}_i (X^{\pi}_{t_i};\theta)-\gb{Z}_{t_i}|^2+\bb{E}|\prom{\ca{G}}_i (X^{\pi}_{t_i};\theta)-\gb{\Gamma}_{t_i} |^2};
\end{align*}
where we used $\gamma = 6K^2$. Then,
\begin{align*}
    \hat{L}(\theta) \ge &~{} \parent{1- Ch}\esp{\widehat{\ca{V}}_{t_i}-\ca{U}_i(X^{\pi}_{t_i};\theta)}^2 \\
    &~{} -\frac{h}{2} \Bigg[\bb{E}|\ca{Z}_i (X^{\pi}_{t_i};\theta)-\gb{Z}_{t_i}|^2+\bb{E}\parent{\int_{\Rd} \big(\gb{U}_{t_i}(y)-\ca{G}_{i}(X^{\pi}_{t_i},y;\theta)\big)^2\lambda(dy)} \Bigg].
\end{align*}
Connecting this bounds using that $\hat{L}(\theta^*)\le\hat{L}(\theta)$ yields that $\forall\theta$,
\begin{align*}
    \parent{1- Ch}\esp{\widehat{\ca{V}}_{t_i}-\ca{U}_i(X^{\pi}_{t_i};\theta^*)}^2 +\frac{h}{2} \bb{E}|\gb{Z}_{t_i}-\ca{Z}_i (X_{t_i},\theta^*)|^2+\frac{h}{2}\bb{E}\parent{\int_{\Rd} \big(\gb{U}_{t_i}(y)-\ca{G}_{i}(X^{\pi}_{t_i},y;\theta^*)\big)^2\lambda(dy)}\\
    \le
    C\esp{\widehat{\ca{V}}_{t_i}-\ca{U}_i(X^{\pi}_{t_i};\theta)}^2 + h\parent{\bb{E}|\gb{Z}_{t_i}-\ca{Z}_i (X_{t_i},\theta)|^2}+h\bb{E}\parent{\norm{\gb{U}_{t_i}(\cdot)-\ca{G}_i(\xscheme{i},\cdot;\theta)}_{L^2(\lambda)}^2}.
\end{align*}
By taking infimum on the right side and $h$ small such that $(1-Ch)\ge \frac{1}{2}$
\begin{align}
   &~{}  \esp{\widehat{\ca{V}}_{t_i}-\widehat{\ca{U}}_i(X^{\pi}_{t_i})}^2 +\frac{h}{2} \bb{E}|\gb{Z}_{t_i}-\widehat{\ca{Z}}_i (X_{t_i})|^2+\frac{h}{2}\bb{E}\parent{\int_{\Rd} \big(\gb{U}_{t_i}(y)-\widehat{\ca{G}}_{i} (X^{\pi}_{t_i},y)\big)^2\lambda(dy)}
  \nonumber\\
  &~{} \qquad   \le C\parent{\ca{E}_i^v + h\ca{E}_i^z + h\ca{E}_i^{\gamma} }.
    \label{eq:cota-imp}
\end{align}
Using this bound on what we found on steps 1 and 2,  we find
\begin{align*}
    \max_{i\in\set{0,...,N-1}}\esp{Y_{t_i}-\ug{i}(X_{t_i}^\pi)}^2\le &~{} C\left[ h + \sum_{i=0}^{N-1} (N\ca{E}_i^v + \ca{E}_i^z + \ca{E}_i^\gamma)  + \sum_{i=0}^{N-1}\bb{E}\parent{\integral|Y_s-Y_{t_i}|^2ds} \right.
    \\ 
    & \qquad + \varepsilon^{Z}(h)+\varepsilon^{\Gamma}(h) +  \esp{g(X_T)-g(X_T^\pi)}^2  \Bigg].
    \nonumber
\end{align*}
Finally, using Proposition \ref{lemma:y-bound}, one ends the proof of \eqref{step34}.

\subsection{Step 4}

    We are going to show some bounds for the terms involving the $\Gamma$ and $U$ components, the same bounds holds for the $Z$ component and are shown in \cite{DBS}. By using (\ref{eq: 4.11g}) on (\ref{eq:ortogonal2}),
\begin{align*}
\bb{E}\parent{\integral |\Gamma_t - \gb{\Gamma}_{t_i}|^2 dt} \le &~{} \bb{E}\parent{\integral |\Gamma_t - \overline{\Gamma}_{t_i}|^2 dt} + 2\lambda(\Rd)\parent{\bb{E}\parent{H_{i+1}^2} - \esp{\bb{E}_i (H_{i+1})}^2} \\
&+ 2h\lambda(\Rd)\bb{E}\parent{\integral f(\Theta_r)^2 dr},
\end{align*}
which implies, after using \eqref{errores0} and  \eqref{lemma:f-int-bound},
\begin{align*}
    \bb{E}\parent{\sum_{i=0}^{N-1}\integral |\Gamma_t - \gb{\Gamma}_{t_i}|^2 dt} &\le \bb{E}\parent{\sum_{i=0}^{N-1}\integral |\Gamma_t - \overline{\Gamma}_{t_i}|^2 dt} + C\sum_{i=0}^{N-1}\parent{\bb{E}\parent{H_{i+1}^2} - \esp{\bb{E}_i (H_{i+1})}^2}+ Ch\\
    &= \varepsilon^{\Gamma}(h) + Ch + C\sum_{i=0}^{N-1}\parent{\bb{E}\parent{H_{i+1}^2} - \esp{\bb{E}_i (H_{i+1})}^2}.
\end{align*}
From \cite{DBS} we get the analogous bound for the $Z$ component, therefore, putting this two together yields
\begin{align}\label{eq:final}
    \bb{E}\parent{\sum_{i=0}^{N-1}\integral \big(|Z_t - \gb{Z}_{t_i}|^2 + |\Gamma_t - \gb{\Gamma}_{t_i}|^2\big) dt}  \le 
    \varepsilon^{Z}(h)+\varepsilon^{\Gamma}(h) + Ch + C\sum_{i=0}^{N-1}\parent{\bb{E}\parent{H_{i+1}^2} - \esp{\bb{E}_i (H_{i+1})}^2}.
\end{align}
This tells us that the next mission in this proof is to give a suitable bound for $\bb{E}\parent{H_{i+1}^2} - \esp{\bb{E}_i (H_{i+1})}^2$. Recall from \eqref{def_H} that $H_{i+1} = Y_{t_{i+1}} - \widehat{\ca{U}}_{i+1} (\xscheme{i+1})$, then
\begin{equation}\label{EHEH}
\begin{aligned}
    \sum_{i=0}^{N-1}\parent{\bb{E}\parent{H_{i+1}^2} - \esp{\bb{E}_i (H_{i+1})}^2} &= \sum_{i=0}^{N-1}\bb{E}(H_{i+1}^2)-\sum_{i=0}^{N-1}\esp{\bb{E}_i(H_{i+1})}^2\\
    &= \esp{Y_{t_N}-\widehat{\ca{U}}_N(\xscheme{N})} + \sum_{i=0}^{N-2}\bb{E}(H_{i+1}^2) - \sum_{i=0}^{N-1}\esp{\bb{E}_i(H_{i+1})}^2\\
    &\le \esp{Y_{t_N}-\widehat{\ca{U}}_N(\xscheme{N})} +\bb{E}(H_{0}^2) + \sum_{i=1}^{N-1}\bb{E}(H_{i}^2) - \sum_{i=0}^{N-1}\esp{\bb{E}_i(H_{i+1})}^2\\
    &=  \esp{g(X_T)-g(X_T^\pi)}^2  +\sum_{i=0}^{N-1} \left( \bb{E}(H_{i}^2) - \esp{\bb{E}_i(H_{i+1})}^2 \right).
\end{aligned}
\end{equation}
From (\ref{eq:4.13}) and (\ref{eq:4.9}) we have an upper and lower bound on $\esp{Y_{t_i}-\widehat{\ca{V}}_{t_i}}^2$. Indeed, first one has
\begin{equation}\label{eq:delta-s4}
 (1-h)\esp{Y_{t_i}-\ug{i}(X^\pi_{t_i})}^2 \leq   \esp{Y_{t_i}-\widehat{\ca{V}}_{t_i}}^2   + \parent{\frac{1}{h}-1}\esp{\ug{i}(X^\pi_{t_i})-\widehat{\ca{V}}_{t_i}}^2.
\end{equation}
Second, we have that for all $\gamma>0$
\begin{align*}
    &\parent{1-h}\, \esp{Y_{t_i}-\widehat{\ca{U}}_{i}(\xscheme{i})}^2 \le   \parent{\frac{1}{h} -1}\esp{\widehat{\ca{U}}_{i}(\xscheme{i})-\widehat{\ca{V}}_{t_i}}^2 + (1+\gamma h)\esp{\bb{E}_i(H_{i+1})}^2 
     \\
    &+(1+\gamma h)\frac{C}{\gamma}\bigg[\underbrace{h^2+ \bb{E}\parent{\integral |Y_s - Y_{t_i}|^2 ds} + h\esp{Y_{t_i}-\widehat{\ca{V}}_{t_i}}^2+ \bb{E}\parent{\integral |Z_s-\gb{Z}_{t_i}|^2 ds} + \bb{E}\parent{\integral |\Gamma_s-\gb{\Gamma}_{t_i}|^2 ds}}_{B_i}\bigg].
\end{align*}
Let us call the expression inside the squared brackets by $B_i$. Subtracting $(1-h) \bb{E}\left|\bb{E}_i(H_{i+1})\right|^2$ and dividing by $(1-h)$,
\[
\begin{aligned}
    \bb{E}(H_i^2) - \esp{\bb{E}_i(H_{i+1})}^2 \le &~{} \frac{1}{h}\esp{\ug{i}(\xscheme{i}) - \widehat{\ca{V}}_{t_i} }^2 + \parent{\frac{h+ \gamma h}{1-h}}\esp{\bb{E}_i(H_{i+1})}^2 +\frac{C}{\gamma}\frac{(1+ \gamma h)}{(1-h)}B_i.
\end{aligned}
\]
For $\gamma = 3C$ and sufficiently small $h$, we can force,
\begin{align*}
    \frac{C}{\gamma}\frac{(1+\gamma h)}{(1-h)} \le \frac{1}{2}\qquad \text{and}\qquad \frac{1}{1-h}\le\frac{1}{2}.
\end{align*}
Hence,
\begin{align*}
    \bb{E}(H_i^2)& - \esp{\bb{E}_i(H_{i+1})}^2 \le \frac{1}{h}\esp{\ug{i}(\xscheme{i}) - \widehat{\ca{V}}_{t_i} }^2 + Ch\esp{\bb{E}_i(H_{i+1})}^2
    +\frac{1}{2}B_i.
\end{align*}
Finally, note that,
\begin{align}\label{eq:N-bound}
\sum_{i=0}^{N-1}\esp{\bb{E}_i(H_{i+1})}^2\le \esp{g(X_T)-g(X_T^\pi)}^2+N\underset{i=0,...,N-1}{\max} \esp{Y_{t_i}-\widehat{\ca{U}}_{i}(\xscheme{i})}^2.
\end{align}
\begin{remark}
	Note that in equation (\ref{eq:N-bound}) appears $N$ multiplying the last term. With the bounds that we have, is impossible to get rid of the $N$, and this is why the $\delta$ improvement mentioned on Remark \ref{delta-remark} will not be of much help.   
\end{remark}
Coming back to \eqref{EHEH},
\[
\begin{aligned}
 \sum_{i=0}^{N-1}\parent{\bb{E}\parent{H_{i+1}^2} - \esp{\bb{E}_i (H_{i+1})}^2}  \leq &~{}  2\esp{g(X_T)-g(X_T^\pi)}^2  + N \sum_{i=0}^{N-1}\esp{\ug{i}(\xscheme{i}) - \widehat{\ca{V}}_{t_i} }^2  \\
 &~{} + Ch N \underset{i=0,...,N-1}{\max} \esp{Y_{t_i}-\widehat{\ca{U}}_{i}(\xscheme{i})}^2 +\frac{1}{2} \sum_{i=0}^{N-1} B_i.
\end{aligned}
\]
Therefore, by plugging this bound in (\ref{eq:final}), noting that $|Y_{t_i}-\widehat{\ca{V}}_{t_i}|^2\le 2|Y_{t_i}-\widehat{\ca{U}}_{i}(\xscheme{i}) |^2 + 2|\widehat{\ca{U}}_{i}(\xscheme{i}) - \widehat{\ca{V}}_{t_i}|^2$, $hN = 1$, and using Lemma \ref{lemma:y-bound}, we have for some $C>0$,
\begin{align*}
     \bb{E}\bigg(\sum_{i=0}^{N-1}\integral \big(|Z_t - \gb{Z}_{t_i}|^2 + |\Gamma_t - \gb{\Gamma}_{t_i}|^2\big) dt\bigg) &\le  C\bigg[\esp{g(X_T)-g(X_T^\pi)}^2 + \varepsilon^{Z}(h)+\varepsilon^{\Gamma}(h) + h\\
    &~{}\qquad+ N \sum_{i=0}^{N-1}\esp{\widehat{\ca{U}}_{t_i}(\xscheme{i}) - \widehat{\ca{V}}_{t_i}}^2 + \underset{i=0,...,N-1}{\max}\esp{Y_{t_i}-\widehat{\ca{U}}_{t_i}(\xscheme{i})}^2\bigg] .
\end{align*}
Now, use (\ref{eq:cota-imp}) together with Lemma \ref{lemma:5.3} to get
\begin{align*}
    \bb{E}\bigg(\sum_{i=0}^{N-1}\integral \big(|Z_t - \gb{Z}_{t_i}|^2 + |\Gamma_t - \gb{\Gamma}_{t_i}|^2\big) dt\bigg) \le & C\bigg[\esp{g(X_T)-g(X_T^\pi)}^2 + \varepsilon^{Z}(h)+\varepsilon^{\Gamma}(h) + h\\
    &\quad+ \sum_{i=0}^{N-1} (N\ca{E}_i^v + \ca{E}_i^z + \ca{E}_i^\gamma)\bigg] .
\end{align*}
Again, recalling (\ref{eq:cota-imp}) using the previous bound and,  
\begin{align*}
    \sum_{i=0}^{N-1}\bb{E}\parent{\integral \Big[|Z_t-\widehat{\ca{Z}}_{i}(\xscheme{i})|^2 + |\Gamma_t-\widehat{\prom{\ca{G}}}_{i}(\xscheme{i})|^2\Big]dt} &\le \sum_{i=0}^{N-1}\bb{E}\parent{\integral \Big[|Z_t-\gb{Z}_{t_i}|^2 + |\Gamma_t-\gb{\Gamma}_{t_i}|^2\Big]dt}\\ 
    &\quad+ \sum_{i=0}^{N-1}h\bb{E}\parent{ \Bigg[|\gb{Z}_{t_i}-\widehat{\ca{Z}}_{i}(\xscheme{i})|^2 + \norm{\gb{U}_{t_i}(\cdot)-\widehat{\ca{G}}_{i}(\xscheme{i},\cdot)}_{L^2(\lambda)}^2\Bigg]dt},
\end{align*}
we conclude that there exist $C>0$, independent of the partition, such that for $h$ sufficiently small, 
\begin{align*}
    &\underset{i=0,...,N-1}{\max}\esp{Y_{t_i}-\widehat{\ca{U}}_{i}(\xscheme{i})}^2 + \sum_{i=0}^{N-1}\bb{E}\parent{\integral \Big[|Z_t-\widehat{\ca{Z}}_{i}(\xscheme{i})|^2 + |\Gamma_t-\widehat{\prom{\ca{G}}}_{i}(\xscheme{i})|^2\Big]dt}\\
    &\le C\left[ h + \sum_{i=0}^{N-1} (N\ca{E}_i^v + \ca{E}_i^z + \ca{E}_i^\gamma)  +\right.
    \varepsilon^{Z}(h)+\varepsilon^{\Gamma}(h) +  \esp{g(X_T)-g(X_T^\pi)}^2  \Bigg].
    \nonumber
\end{align*}
This ends the proof of the following Theorem:
\begin{theorem}\label{MT1}

Under \textbf{(C)}, there exists a constant $C$ independent of the partition such that for sufficiently small $h$,
\begin{align*}
    &\underset{i=0,...,N-1}{\max}\esp{Y_{t_i}-\widehat{\ca{U}}_{i}(\xscheme{i})}^2 + \sum_{i=0}^{N-1}\bb{E}\parent{\integral \Big[ |Z_t-\widehat{\ca{Z}}_{i}(\xscheme{i})|^2 + |\Gamma_t-\widehat{\prom{\ca{G}}}_{i}(\xscheme{i})|^2\Big] dt}\\
     &~{} \qquad \le
    C\left[ h + \sum_{i=0}^{N-1} (N\ca{E}_i^v + \ca{E}_i^z + \ca{E}_i^\gamma) +\right.
    \varepsilon^{Z}(h)+\varepsilon^{\Gamma}(h) +  \esp{g(X_T)-g(X_T^\pi)}^2  \Bigg],
    \nonumber
\end{align*}
with $\ca{E}_i^v$,  $\ca{E}_i^z$  and $\ca{E}_i^\gamma$ given in \eqref{errores}, and $\varepsilon^{Z}(h)$ and $\varepsilon^{\Gamma}(h)$ defined in \eqref{errores0}.
\end{theorem}

Note that the terms involving the NNs $\ca{E}_i^v$,  $\ca{E}_i^z$  and $\ca{E}_i^\gamma$, can be made arbitrarily small, in view of Theorem \ref{uat} and \eqref{NN_L2norm}.
The challenge here, and therefore in almost every DL algorithm, is that we don't know how many hidden layers
and units per layer we need to achieve a fixed tolerance, we only can ensure the existence of such NN architecture.\\

We state some remarks from the proof.

\begin{remark}
The main difficulty of the adaptation of the proof given in \cite{DBS}, was to give a useful definition of the third NN with the mission of approximate the non local component. This was problematic because we have two options, the first is to define the NN to approximate the whole integral 
        \[
            \int_{\Rd}[u(t_i,X_{t_i}^{\pi}+\beta(X_{t_i}^{\pi},y))-u(t_i,X_{t_i}^{\pi})]\lambda(dy),
        \]
        which seems intuitive because this will lead our third NN to approximate the nonlocal part of the PIDE and, therefore, receive one parameter: $\xscheme{i}$. But, we also need to approximate or been able to calculate the stochastic integral
        \[
            \int_{\Rd}[u(t_i,X_{t_i}^{\pi}+\beta(X_{t_i}^{\pi},y))-u(t_i,X_{t_i}^{\pi})]\bar{\mu}\parent{(t_i,t_{i+1}],dy},
        \]
        that cannot be done by just knowing the first integral. To overcome this issue, we proposed the idea to approximate just what it is inside the integrals and solve the problem of actually integrate this function with other tools.
\end{remark} 
\begin{remark}       
 The non local part of the PIDE (\ref{eq: pide}) makes us add a Lévy process, which is a canonical tool when dealing with non local operators such as the one that appears on equation (\ref{eq: pide}). This addition results in the natural definition of analogous objects from \cite{DBS} such as the $\Gamma, \bar{\Gamma}$ components for the nonlocal case.     
\end{remark}

 \begin{remark}
 Because of the finite character of the measure $\lambda$, the case of the Fractional Laplacian mentioned in the introduction is not contained in Theorem \ref{MT1}. We hope to extend our results to this case in a forthcoming result. 
 \end{remark}

\begin{remark}
    The result of the theorem states that the better we can approximate $v_i, z_i, \gamma_i$ by NN architectures, the better we can approximate $(Y_{t_i}, Z_{t_i}, \Gamma_{t_i})$ by $(\ug{i}(\xscheme{i}),\widehat{\ca{Z}}_{i}(\xscheme{i}),\prom{\widehat{\ca{G}}}_{i}(\xscheme{i}))$. 
\end{remark}

\subsection{Optimization step of the algorithm \label{sect:opti}} 

In this subsection we give a brief but complete description of how to compute the loss function from Algorithm \ref{algorithm1}.As usual, we extend the computation of the loss function shown on \cite{DBS} to our non local case. For simplicity assume that $\lambda$ is a probability measure absolutely continuous
with respect to Lebesgue measure. As we will see, several simulation of the Lévy process $(X_t)_t$ are needed.\\

Given $\ug{i+1}$, we need to minimize $L_i(\cdot)$ and define the NNs for step $i$. Recall the Definition \ref{eq:loss}, the idea is to write the expected value from the loss function as an average of simulations. Let $M\in\bb{N}$ and $I = \set{1,..., M}$, generate simulations $\set{x_k: k\in I}$, $\set{w_k:\in I}$ of $\xscheme{i}$ and $\Delta W_i$ respectively. Then,
\begin{align*}
    L_i(\theta)\approx\frac{1}{M}\sum_{k\in I}  \big(\ug{i+1}(x_k)-F(t_i,x_k,\ca{U}_i(x_k;\theta),\ca{Z}_i(x_k;\theta),\ca{G}_i(x_k,\cdot;\theta),h,w_i)\big)^2.
\end{align*}
Note that we are using an Euler scheme on the simulations of $(X_t)_t$, nevertheless, there exists other methods depending on the structure of the diffusion, see \cite{rk} and \cite{jump-adapted}. Recall that $F$ needs two different integrals of $\ca{G}_i(x_k,\cdot;\theta)$, to approximate these values let $L\in\bb{N}$ and $J=\set{1,...,L}$  and consider, for every $k\in I$, simulations $\set{y^k_{l}: l\in J}$ of a random variable $Y\sim\lambda$, here is important the finitness of the measure. Then, the quantities we need can be computed as follows,
\begin{align*}
    \int_{\Rd}\ca{G}_i(x_k,y;\theta)\lambda(dy) &\approx \frac{1}{L}\sum_{l\in J} \ca{G}_i(x_k,y^k_{l};\theta)\\
    \int_{\Rd}\ca{G}_i(x_k,y;\theta)\bar{\mu}\parent{(t_i,t_{i+1}],dy} &= \int_{t_i}^{t_{i+1}}\int_{\Rd}\ca{G}_i(x_k,y;\theta)\mu\parent{dt,dy} - \integral\int_{\Rd} \ca{G}_i(x_k,y;\theta)dt\lambda(dy) \\
    &\approx
    \sum_{t_i\le s<t_{i+1}} \ca{G}_i(x_k,\Delta P_s;\theta)\ind{\Rd} (\Delta P_s) - \frac{h}{L}\sum_{l\in J} \ca{G}_i(x_k,y^k_{l};\theta).
\end{align*}
Therefore, provided we can simulate: trajectories of $(X_t)_t$ and $(W_t)_t$, realizations of $Y\sim \lambda$ and the compound Poisson process $(P_s)_s$, we can minimize $L_i$, find the optimal $\theta^*$ and define
\begin{align*}
    (\ug{i},\widehat{\ca{Z}}_{i},\widehat{\ca{G}}_{i})=(\ca{U}_i(\cdot;\theta^*),\ca{Z}_i(\cdot;\theta^*),\ca{G}_i(\cdot,\circ;\theta^*)).
\end{align*}
\begin{remark}
    The nonlocal term in equation (\ref{eq: pide}) adds complexity not only in the proof of the consistency of the algorithm but in the algorithm itself. As we saw, it is key that the measure $\lambda$ is finite as well as the capability to simulate integrals with respect to Poisson random measures and trajectories of the Lévy process. The implementation of this method and an extension to PIDEs with more general integro-differential operators, such as fractional Laplacian, are left to future work. 
\end{remark}


\addcontentsline{toc}{section}{References}
\begin{thebibliography}{99}


\bibitem{allaire}
Grégoire Allaire, \emph{Numerical Analysis and Optimization An introduction to mathematical modeling and numerical simulation}, Oxford University Press; Illustrated edition (July 19, 2007), 472 pages. ISBN-10 : 9780805839852.

\bibitem{state-art-dl}
Md Zahangir Alom, Tarek M. Taha, Chris Yakopcic, Stefan Westberg, Paheding Sidike, Mst Shamima Nasrin, Mahmudul Hasan, Brian C. Van Essen, Abdul A. S. Awwal and Vijayan K. Asari, \emph{A State-of-the-Art Survey on Deep Learning Theory and Architectures}, Electronics 2019, 8(3), 292; https://doi.org/10.3390/electronics8030292.

\bibitem{brain-tumor}
Ali Mohammad Alqudah, Hiam Alquraan, Isam Abu Qasmieh, Amin Alqudah, and Wafaa Al-Sharu, \emph{Brain Tumor Classification Using Deep Learning Technique - A Comparison between Cropped, Uncropped, and Segmented Lesion Images with Different Sizes}, International Journal of Advanced Trends in Computer Science and Engineering, Volume 8, No.6, 2019.

\bibitem{applebaum}
David Applebaum, \emph{Lévy Processes And Stohastic Calculus}, Cambridge Studies In Advanced Mathematics, 2nd Edition, (April 1, 2009).
ISBN-10: 0521738652.


\bibitem{pardoux}
Guy Barles, Rainer Buckdahn and Etienne Pardoux, \emph{Backward Stochastic Differential equations and integral-partial differential equations}, Stochastics and Stochastics Reports, Vol. 60, pp. 57-83, 1996.

\bibitem{erwin3}
Guy Barles, Olivier Ley, and Erwin Topp, \emph{Lipschitz Regularity For Integro-Differential Equations With Coercive Hamiltonians And Applications To Large Time Behavior,} Nonlinearity, Volume 30, Number 2 (2017), arXiv:1602.07806 [math.AP].

\bibitem{astronomy-app1}
Dalya Baron, \emph{Machine Learning In Astronomy: A Practical Overview}, arXiv:1904.07248v1 [astro-ph.IM] 15 Apr 2019.


\bibitem{intro1}
Christian Beck, Fabian Hornung, Martin Hutzenthaler, Arnulf Jentzen, and Thomas Kruse, \emph{Overcoming the curse of dimensionality in the numerical approximation of Allen–Cahn partial differential equations via truncated full-history recursive multilevel Picard approximations}, Accepted in J. Numer. Math. arXiv:1907.06729 [math.NA], 2019.


\bibitem{intro5}
Christian Beck, Weinan E, and Arnulf Jentzen, \emph{Machine learning approximation algorithms for high-dimensional fully nonlinear partial differential equations and second-order backward stochastic differential equations},  J. Nonlinear Sci. 29 (2019), 1563--1619, arXiv:1709.05963v1 [math.NA], 2017.

\bibitem{advection-dispersion}
 David A. Benson, Stephen W. Wheatcraft, and Mark M. Meerschaert, \emph{Application of a fractional advection-dispersion equation}, Water Resources Research, Vol. 36, No. 6, Pages 1403--1412, June 2000.


\bibitem{erwin1}
Isabeau Birindelli, Giulio Galise, And Erwin Topp, \emph{Fractional Truncated Laplacians: Representation Formula, Fundamental Solutions And Applications}, arXiv:2010.02707 [math.AP], 2020.


\bibitem{bruno}
Bruno Bouchard, Romuald Elie. \emph{Discrete time approximation of decoupled Forward-Backward SDE with jumps}. Stochastic Processes and their Applications, Elsevier, 2008, 118 (1), pp. 53--75. ffhal00015486.

\bibitem{physics-dl1}
Dimitri Bourilkov, \emph{Machine and Deep Learning Applications in Particle Physics}, International Journal of Modern Physics A 34(35):1930019
DOI: 10.1142/S0217751X19300199. arXiv:1912.08245v1 [physics.data-an].

\bibitem{rk}
Evelyn Buckwar, Martin G. Riedler, \emph{Runge–Kutta methods for jump–diffusion differential equation}, Journal of Computational and Applied Mathematics 236 (2011) 1155–1182.

\bibitem{CS}
Luis Caffarelli and Luis Silvestre, \emph{An Extension Problem Related to the Fractional Laplacian}, Comm. PDE Vol. 32, 2007 Issue 8 pp. 1245--1260.

\bibitem{asset-returns}
P. Carr, H. Geman, D.B. Madan, and M. Yor, \emph{The fine structure of asset returns: An empirical investigation}, Journal of Business, 75: 305--332, 2002.



\bibitem{financial-app}
Rama Cont, and Peter Tankov, \emph{Financial Modelling with Jump Processes}, Chapman and Hall/CRC; 1st edition (December 30, 2003). ISBN-10: 1584884134, 552 pp.



\bibitem{erwin2}
Gonzalo Dávila And Erwin Topp, \emph{The Nonlocal Inverse Problem Of Donsker And Varadhan}, arXiv:2011.13295 [math.AP], 2020.


\bibitem{non-local-guide}
Marta D’Elia, Qiang Du, Christian Glusa, Max Gunzburger, Xiaochuan Tian and Zhi Zhou, \emph{Numerical methods for nonlocal and fractional models}, Acta Numerica (2020), pp. 1--124.


\bibitem{lukasz}
Łukasz Delong, \emph{Backward Stochastic Differential Equations with Jumps and Their Actuarial and Financial Applications}, EEA series, Springer-Verlag London, 2013. doi:10.1007/978-1-4471-5331-3, 288+X pp.


\bibitem{DPV}
Eleonora Di Nezza, Giampiero Palatucci, and Enrico Valdinoci, \emph{Hitchhiker'{}s guide to the fractional Sobolev spaces}, Bulletin des Sciences Mathématiques
Volume 136, Issue 5, July--August 2012, Pages 521--573. 


\bibitem{malliavin}
Giulia Di Nunno, Bernt Øksendal
Frank Proske, \emph{Malliavin Calculus for Levy Processes with Applications to Finance}, Universitext Springer-Verlag Berlin Heidelberg 2009. DOI 10.1007/978-3-540-78572-9, XIV+418 pp.

\bibitem{perodynamic}
Qiang Du, and Xiaochuan Tian, \emph{Stability Of Nonlocal Dirichlet Integrals And Implications For Peridynamic Correspondence Material Modeling}, SIAM J. Appl. Math. (2018) Vol. 78, No. 3, pp. 1536--1552, arXiv:1710.05119 [physics.comp-ph].





\bibitem{image-processing}
Guy Gilboa, and Stanley Osher, \emph{Nonlocal Operators With Applications To Image Processing}, Multiscale Modeling Simulation, 7: 1005--1028, 2008.



\bibitem{HJE} Han, J., Jentzen, A., E, W., \emph{Solving high-dimensional partial differential equations using deep learning}. Proc. Natl. Acad. Sci. 115 (2018), 8505--8510.

\bibitem{Hornik}
Kurt Hornik, \emph{Approximation Capabilities of Multilayer
Feedforward Networks}, Neural Networks, Vol. 4, pp. 251-257. 1991 


\bibitem{DBS}
Come Hure, Huyen Pham, and Xavier Warin, \emph{Deep Backward Schemes
For High-Dimensional Nonlinear PDE's}, Math. Comp. 89 (2020), 1547--1579.

\bibitem{intro2}
Martin Hutzenthaler, Arnulf Jentzen, Thomas Kruse, Tuan Anh Nguyen, and Philippe von Wurstemberger, \emph{Overcoming the curse of dimensionality in the numerical approximation of semilinear parabolic partial differential equations}, arXiv:1807.01212 [math.PR], 2018. Accepted in Proc. Roy. Soc. A.




\bibitem{levy-pdes-related}
Benjamin Jourdain, Sylvie Méléard, and Wojbor A. Woyczynski, \emph{Nonlinear SDEs driven by Lévy processes and related PDEs}, ALEA Lat. Am. J. Probab. Math. Stat. 4 (2008), 1--29. arXiv:0707.2723, 2007.


\bibitem{jump-adapted}
Arturo Kohatsu-Higa, Peter Tankov, \emph{Jump-adapted discretization schemes for Lévy-driven SDEs}, Stochastic Processes and their Applications, Volume 120, Issue 11, 2010, Pages 2258-2285, ISSN 0304-4149, https://doi.org/10.1016/j.spa.2010.07.001.



\bibitem{torres}
Antoine Lejay, Ernesto Mordecki, and Soledad Torres, \emph{Numerical approximation of Backward Stochastic Differential Equations with Jumps}, 2007. ffinria-00357992v2.

\bibitem{Leshno}
Moshe Leshno, I. Vladimir Ya. Lin, Allan Pinkus, And Shimon Schocken, \emph{Multilayer Feedforward Networks With a Nonpolynomial
Activation Function Can Approximate Any Function}, Neural Networks, Vol. 6, pp. 861--867 (1993).

\bibitem{medical-survey}
Geert Litjens, Thijs Kooi, Babak Ehteshami Bejnordi, Arnaud Arindra Adiyoso Setio, Francesco Ciompi, Mohsen Ghafoorian, Jeroen A.W.M. van der Laak, Bram van Ginneken, and Clara I. Sanchez, \emph{A Survey on Deep Learning in Medical Image Analysis}, Medical Image Analysis
Volume 42, December 2017, Pages 60--88, arXiv:1702.05747v2 [cs.CV] 4 Jun 2017.




\bibitem{intro4}
 Martin Magill, Andrew M. Nagel and Hendrick W. de Haan, \emph{Neural Network Solutions to Differential Equations in Non-Convex Domains: Solving the Electric Field in the Slit-Well Microfluidic Device}, Phys. Rev. Research 2, 033110 -- Published 21 July 2020. ArXiv:2004.12235v1 [physics.comp-ph], 2020.

\bibitem{light-curves}
A. Mahabal, K. Sheth, F. Gieseke, A. Pai, S. G. Djorgovski, A. J. Drake, M. J. Graham, and CSS/CRTS/PTF Teams, \emph{Deep-Learnt Classification of Light Curves}, 2017 IEEE Symposium Series on Computational Intelligence (SSCI), Honolulu, HI, 2017, pp. 1-8, doi: 10.1109/SSCI.2017.8280984. arXiv:1709.06257v1 [astro-ph.IM].

\bibitem{street-style}
Kevin Matzen, Kavita Bala, Noah Snavely, \emph{StreetStyle: Exploring world-wide clothing styles from millions of photos,} arXiv:1706.01869v1 [cs.CV] 6 Jun 2017.



\bibitem{pitts}
Warreb S. Mcculloch And Walter Pitts, \emph{A Logical Calculus Of The Ideas Immanent In
Nervous Activity}, Bulletin of Mathematical Biophysics, Vol. 5, pp. 115-133, 1943.






\bibitem{perceptron}
F. Rosenblatt, \emph{The Perceptron: A Probabilistic Model For
Information Storage And Organization
In The Brain}, Psychological Review
Vol. 65, No. 6, 1958.


\bibitem{intro3}
Justin Sirignano, Konstantinos Spiliopoulos, \emph{DGM: A deep learning algorithm for solving partial differential equations}, Journal of Computational Physics Volume 375, 15 December 2018, Pages 1339--1364, arXiv:1708.07469 [q-fin.MF], 2017.


\bibitem{fractional-operator}
Pablo Raúl Stinga, \emph{User’s guide to the fractional Laplacian and the method of semigroups}, in: Fractional Differential Equations,
Walter de Gruyter GmbH \& Co KG, pp. 235--266, arXiv:1808.05159 [math.AP], 2018. 



\bibitem{hydrodynamics}
E. Tadmor and C. Tan, \emph{Critical thresholds in flocking hydrodynamics with non-local alignment},  
Philosophical Transactions of the Royal Society of London A: Mathematical, Physical and
Engineering Sciences, 372 (2014), 20130401.


\bibitem{qst}
Giacomo Torlai, Guglielmo Mazzola, Juan Carrasquilla, Matthias Troyer, Roger Melko, and Giuseppe Carleo, \emph{Neural-network quantum state tomography for many-body systems}, Nature Physics volume 14, pages 447--450 (2018), arXiv:1703.05334v2 [cond-mat.dis-nn].





\bibitem{origin-dl}
Haohan Wang, Bhiksha Raj, \emph{On the Origin of Deep Learning}, arXiv:1702.07800v4 [cs.LG] 3 Mar 2017.





\bibitem{zha04}
 Jianfeng Zhang, \emph{A Numerical Scheme For BSDES}, Annals of Applied Probability 2004, Vol. 14, No. 1, 459–488.

\bibitem{quasilinear}
Xicheng Zhang, \emph{Stochastic Functional Differential Equations
Driven By Levy Processes And Quasi-Linear Partial, Integro-Differential Equation}, Ann. Appl. Probab.
Volume 22, Number 6 (2012), 2505-2538. arXiv:1106.3601 [math.PR]. 





















 

















\end{thebibliography}
\end{document}